\documentclass[a4paper,11pt]{amsart}
\usepackage{amsthm}
\usepackage{amsmath}
\usepackage{amssymb}
\usepackage{amscd}                 
\usepackage[all]{xy}               
\usepackage{stmaryrd}              
\usepackage{a4wide}
\usepackage[toc,page]{appendix}
\usepackage[utf8]{inputenc}
\usepackage[english]{babel}
\usepackage{booktabs} 
\usepackage{todonotes}
\usepackage[fixlanguage]{babelbib}
\usepackage{tikz-cd}
\selectbiblanguage{english}

\newtheorem{theorem}{Theorem}

\newtheorem{corollary}{Corollary}
\newtheorem{lemma}{Lemma}

\newtheorem{proposition}{Proposition}

\newtheorem{remark}{Remark}
\newtheorem{example}{Example}

\newcommand{\Gal}{\mathop{\mathrm{Gal} }\nolimits}

\newcommand{\SL}{\mathop{\mathrm{SL} }\nolimits}
\newcommand{\GL}{\mathop{\mathrm{GL} }\nolimits}

\newcommand{\tors}{\mathrm{tors}}
\newcommand{\wild}{\mathrm{wild}}

\title[On the torsion growth in quadratic number fields for elliptic curves]{On the torsion growth in quadratic number fields for elliptic curves defined over the rationals}

\author[S. Arias-de-Reyna]{Sara Arias-de-Reyna}
\address{Departamento de \'Algebra, Facultad de Matem\'aticas and IMUS, Universidad de Sevilla. Avda. Reina Mercedes s/n, 41012 Sevilla, Spain.}
\email{sara\_arias@us.es}
\author[M. Pineda-Mart'in]{Miguel Pineda-Mart\'in}
\address{Departamento de \'Algebra, Facultad de Matem\'aticas and IMUS, Universidad de Sevilla. Avda. Reina Mercedes s/n, 41012 Sevilla, Spain.}
\email{miguelpinedamartin@gmail.com}
\author[J.M. Tornero]{Jos\'e M. Tornero}
\address{Departamento de \'Algebra, Facultad de Matem\'aticas and IMUS, Universidad de Sevilla. Avda. Reina Mercedes s/n, 41012 Sevilla, Spain.}
\email{tornero@us.es}
\thanks{This work was supported by IMUS-Maria de Maeztu grant CEX2024-001517-M - Apoyo a Unidades de Excelencia María de Maeztu, funded by MICIU/AEI/ 10.13039/501100011033, grant PID2024-156912N funded by MICIU/AEI/10.13039/501100011033 and ERDF/EU and grants SOL2024-31596 and SOL2024-31708 from the Plan Propio de Investigación y Transferencia of the University of Sevilla, cofunded by the EU - Ministerio de Hacienda y Función Pública - Fondos Europeos - Junta de Andalucía – Consejería de Universidad, Investigación e Innovación”.}
\date{\today}

\begin{document}

\begin{abstract}
Given an elliptic curve defined over the field of rational numbers, it is known how its torsion subgroup may grow when we make a base change to a quadratic number field. In this paper we consider the inverse question: if we have the elliptic curve defined over the rationals and we know how the torsion subgroup grows, what can we say about the field? Our main result gives an explicit relationship between the primes dividing the conductor of the curve and the conductor of the extension as a first approach to a better understanding of this problem.
\end{abstract}

\maketitle

\section{Introduction: The problem}

The following notations and conventions will be used throughout the paper:
\begin{itemize}
    \item We will write $\mathcal{C}_r$ for the cyclic group of order $r$.
    \item As it is customary in the context of elliptic curves, the groups are usually written in additive notation.
    \item Given an elliptic curve $E$ defined over a number field $K$, we will write $E(K)$ for the group of points of $E$ with coordinates on $K$ and $E_\tors(K)$ for its torsion subgroup (including the case $K=\mathbb{Q}$).
    \item We will write $o(P)$ for the order of the point $P$ on the group $E(K)$.
    \item Whenever we consider a quadratic number field written as $K = \mathbb{Q} (\sqrt{d})$ we will assume $d$ is a square-free integer.
    \item Examples are taken from \cite{lmfdb} and labeled accordingly.
\end{itemize}

\ 

Let $E/\mathbb{Q}$ be an elliptic curve defined over $\mathbb{Q}$. Let $\ell$ be a prime number and let us write:
\begin{itemize}
    \item $E[\ell]$ for the group of $\ell$-torsion points on $E(\overline{\mathbb{Q}})$, where $\overline{\mathbb{Q}}$ denotes an algebraic closure of $\mathbb{Q}$.
    \item $\mathbb{Q}(E[\ell])$ for the extension generated by the coordinates of the points of $E[\ell]$. 
    \item $E(K)[\ell]$ for the group of $K$-rational $\ell$-torsion points for every field $K$ (including the case $K = \mathbb{Q}$). 
\end{itemize} 

This paper deals with properties of the torsion subgroup of elliptic curves defined over the rationals under quadratic field extensions. Let then $E/\mathbb{Q}$ be an elliptic curve, $K$ a quadratic number field. The first papers addressing the comparison between $E_\tors(\mathbb{Q})$ and $E_\tors(K)$ were \cite{GonTorI, GonTorII, Najman}. In particular, \cite[Theorem 2]{GonTorI} states the following:

\begin{theorem}\label{GT}
With the previous notations, we have the following table:
$$
\renewcommand\arraystretch{1.3}
\begin{array}{|c|l|}
\hline
E_\tors(\mathbb{Q}) & \text{Groups that can appear as } E_\tors(K) \\
\hline
\mathcal{C}_1 & \left\{ \mathcal{C}_1\,,\, \mathcal{C}_3 \,,\, \mathcal{C}_5\,,\, \mathcal{C}_7\,,\, \mathcal{C}_9 \right\} \\
\hline
\mathcal{C}_2 & \left\{ \mathcal{C}_2\,,\, \mathcal{C}_4 \,,\, \mathcal{C}_6\,,\, \mathcal{C}_8\,,\, \mathcal{C}_{10}\,,\, \mathcal{C}_{12}\,,\, \mathcal{C}_{16}\,,\, \mathcal{C}_2 \times \mathcal{C}_{2}\,,\, \mathcal{C}_2 \times \mathcal{C}_{6}\,,\, \mathcal{C}_2 \times \mathcal{C}_{10} \right\} \\ 
\hline
\mathcal{C}_3 & \left\{ \mathcal{C}_3, \; \mathcal{C}_{15}, \; \mathcal{C}_3 \times \mathcal{C}_3 \right\} \\
\hline
\mathcal{C}_4 & \left\{ \mathcal{C}_4 \,,\,\mathcal{C}_8\,,\, \mathcal{C}_{12}\,,\, \mathcal{C}_2 \times \mathcal{C}_{4}\,,\, \mathcal{C}_2 \times \mathcal{C}_{8}\,,\, \mathcal{C}_2 \times \mathcal{C}_{12}\,,\, \mathcal{C}_4 \times \mathcal{C}_4\right\} \\
\hline
\mathcal{C}_5 & \left\{ \mathcal{C}_5, \; \mathcal{C}_{15} \right\} \\
\hline
\mathcal{C}_6& \left\{ \mathcal{C}_6, \; \mathcal{C}_{12}, \; \mathcal{C}_2 \times \mathcal{C}_6, \; \mathcal{C}_3 \times \mathcal{C}_6 \right\} \\
\hline
\mathcal{C}_7 & \left\{ \mathcal{C}_7 \right\} \\
\hline
\mathcal{C}_8 & \left\{ \mathcal{C}_8, \; \mathcal{C}_{16}, \; \mathcal{C}_2 \times \mathcal{C}_8 \right\} \\
\hline
\mathcal{C}_9 & \left\{ \mathcal{C}_9 \right\} \\
\hline
\mathcal{C}_{10} & \left\{ \mathcal{C}_{10}, \; \mathcal{C}_2 \times \mathcal{C}_{10} \right\} \\
\hline
\mathcal{C}_{12} & \left\{ \mathcal{C}_{12}, \; \mathcal{C}_2 \times \mathcal{C}_{12} \right\} \\
\hline
\mathcal{C}_2 \times \mathcal{C}_2 & \left\{ \mathcal{C}_2 \times \mathcal{C}_{2}\,,\, \mathcal{C}_2 \times \mathcal{C}_{4}\,,\, \mathcal{C}_2 \times \mathcal{C}_{6}\,,\, \mathcal{C}_2 \times \mathcal{C}_{8}\,,\,  \mathcal{C}_2 \times \mathcal{C}_{12}\right\} \\
\hline
\mathcal{C}_2 \times \mathcal{C}_4 & \left\{ \mathcal{C}_2 \times \mathcal{C}_4, \; \mathcal{C}_2 \times \mathcal{C}_8, \; \mathcal{C}_4 \times \mathcal{C}_4 \right\} \\
\hline
\mathcal{C}_2 \times \mathcal{C}_6 & \left\{ \mathcal{C}_2 \times \mathcal{C}_6, \mathcal{C}_2 \times \mathcal{C}_{12} \right\} \\
\hline
\mathcal{C}_2 \times \mathcal{C}_8 & \left\{ \mathcal{C}_2 \times \mathcal{C}_8 \right\} \\
\hline
\end{array}
$$
\end{theorem}

\ 

In the same reference, three further problems were stated. Problems 1 and 3 from this list were solved in \cite{GonTorII, Najman} independently. The origin of our research was trying to answer the problem originally stated in \cite[Section 4]{GonTorI} as Problem 2.

\ 

\noindent \textbf{Problem:} Is there a precise (and easy) description of which are the possible quadratic number fields $K = \mathbb{Q}(\sqrt{d})$ with $E_\tors(\mathbb{Q}) \neq E_\tors(K)$, ideally in terms of some invariant(s) of the curve?

\ 

The extensive calculations from \cite{GonTorII} were the starting point of our research, as we tried to prove conditions on the integer $d$ (alternatively, on the field $K$), assuming we do have growth of torsion subgroups. More specifically, we were interested on sieving the possible prime divisors of $d$ from invariants of $E$.

\ 

Let now $E/\mathbb{Q}$ be an elliptic curve with conductor $N_E$ and $K = \mathbb{Q} (\sqrt{d})$ a quadratic number field with $E_\tors(\mathbb{Q}) \neq E_\tors(K)$. Assume $P \in E_\tors(K) \setminus E_\tors(\mathbb{Q})$ with $o(P)=n$. We know from \cite{KenMom, Kamienny} the primes $\ell$ which can divide $n$ are $\{ 2,3,5,7 \}$.

\ 

Since $P\in E[n]\cap E(K)$, we have that the coordinates of $P$ belong to $K \cap \mathbb{Q}(E[n])$. But $K \cap \mathbb{Q}(E[n])$ must be either $\mathbb{Q}$ or $K$, since $K/\mathbb{Q}$ is a quadratic extension. As we are assuming that $P \notin E(\mathbb{Q})$, it follows that $K \subset \mathbb{Q}(E[n])$. On the other hand, we know from the Neron--Ogg--Shafarevich Criterion that $\mathbb{Q}(E[n])/\mathbb{Q}$ only ramifies at primes dividing $N_E$ or $n$ \cite[Theorem 7.1]{Silverman}. Therefore, we can state:

\begin{proposition}\label{prop:1}
Let $E/\mathbb{Q}$ be an elliptic curve. If $K$ is a quadratic number field such that $E_\tors(\mathbb{Q}) \not= E_\tors(K)$, then the primes ramifying in $K$ belong to the set $\{2, 3, 5, 7\} \cup \{p \; : \; p\vert N_E\}$.
\end{proposition}

\begin{remark}\label{rem:1}
Under the hypothesis of the proposition, if $E$ has good reduction at $p \in \{2,3,5,7\}$ (i.e. $p \nmid N_E$), then $p$ ramifying in $K$ implies $E_\tors(\mathbb{Q})[p] \neq E_\tors(K)[p]$.
\end{remark}

Let us now write $K = \mathbb{Q}(\sqrt{d})$. Since every prime dividing $d$ ramifies in $K$, we have the following version of the previous proposition, more appropriate for our goals.

\begin{proposition}\label{prop:2}
Let $E/\mathbb{Q}$ be an elliptic curve. If $K$ is a quadratic number field such that $E_\tors(\mathbb{Q}) \not= E_\tors(K)$, then we can write $K=\mathbb{Q}(\sqrt{d})$, where the primes dividing $d$ belong to the set $\{2, 3, 5, 7\} \cup \{p: p\vert N_E\}$.
\end{proposition}

\ 

A first property we looked into after this was the following natural sequel: If $p \in \mathbb{Z}$ is a prime such that $p|d$, can we always say $p|N_E$? As stated before, under these circumstances, either $p|N_E$ or $E_\tors(\mathbb{Q})[p] \neq E_\tors(K)[p]$, so it is natural to look at points defined over $K$ of prime order. Our strategy of choice for that study has been the so--called mod $\ell$ Galois representations, which we review now briefly.

\ 

Let $\ell$ be a prime integer. As it is well--known \cite{Serre}, the action of the absolute Galois group $\mathrm{G}_\mathbb{Q}=\Gal(\overline{\mathbb{Q}}/\mathbb{Q})$ on $E[\ell]$ defines a mod $\ell$ Galois representation
$$
\overline{\rho}_{E,\ell}:\mathrm{G}_\mathbb{Q} \longrightarrow \mathrm{Aut}(E[\ell]) \simeq \GL_2(\mathbb{F}_\ell).
$$

The extension $\mathbb{Q} (E[\ell])/ \mathbb{Q}$ is Galois, with $\Gal(\mathbb{Q}(E[\ell])/ \mathbb{Q}) \simeq \overline{\rho}_{E,\ell}(\mathrm{G}_\mathbb{Q})$. In this context, therefore, the subgroup lattice of $\GL_2(\mathbb{F}_\ell)$ offers relevant information about the field extensions we are interested in.

\ 

Fix then $E/\mathbb{Q}$ an elliptic curve, $K=\mathbb{Q}(\sqrt{d})$ a quadratic number field such that $E_\tors(\mathbb{Q}) \not= E_\tors(K)$ and $\ell$ a prime such that $\ell \nmid N_E$. 

We will study, case by case, how the situation unfolds for the possible primes $\ell \in \{ 2,3,5,7 \}$ separatedly. 

\ 

\noindent \boxed{\textbf{The case $\ell=2$.}} (Sections 2 \& 3). In this case we were able to prove the following.

\begin{theorem}\label{l2}
Let $E/\mathbb{Q}$ be an elliptic curve, and $K$ a quadratic number field such that $E_\tors(\mathbb{Q}) \not= E_\tors(K)$. If $2\vert d$, then $2\vert N_E$.
\end{theorem}

The proof of this theorem encompasses Propositions \ref{2strict}, \ref{N=4}, \ref{N=8} and \ref{N=16}, which also yield some collateral results as Corollary \ref{cor16}. It is the most complicated result, and it must be broken, to begin with, in two parts, illustrating two different situations for the torsion growth.

\ 

We know a new torsion point must appear when moving from $\mathbb{Q}$ to $K$. It may happen that some new torsion point $P\in E(K)$ appears with no nontrivial multiple of $P$ belonging to $E(\mathbb{Q})$ (this will be called the \textit{strict} case). Or it may happen that, for each new torsion point $P\in E(K)$, there is a nontrivial multiple $mP$ belonging to $E(\mathbb{Q})$ (this will be the \textit{mixed} case). 

For example, if we have $E_\tors(\mathbb{Q})\simeq \mathcal{C}_2$ and $E_\tors(K)\simeq \mathcal{C}_2 \times \mathcal{C}_2$, it means that a new point of order $2$ has appeared, and we would be in the first case. However, if $E_\tors(K)\simeq \mathcal{C}_4$, it means that a point of $4$-torsion, $P$, has appeared such that $2P\in E_\tors(\mathbb{Q})$. This would be an example of the second case. 

Note that, in the strict cases, we can always assume that the order of $P$ is a prime number $\ell$ since, if such a point $P$ appears, some multiple of it will have prime order, and under our assumption it can not be contained in $E_\tors(\mathbb{Q})$. 

The mixed case will be the most involved and it is distinctive of the $\ell=2$ case. Indeed, this situation can in principle happen with points of any prime order in $\{2, 3, 5, 7\}$ but, under our hypothesis, at least one new torsion point in this case must have an order that is a multiple of $2$, as shown in Theorem \ref{GT}.

\ 

\noindent \boxed{\textbf{The case $\ell=3$.}} (Section 4). This prime has a very specific behaviour as it is the only one who allows \textit{unexpected} points to appear. By that we mean we know for sure that $3|d$ does \textit{not} imply $3|N_E$. An example of this is the curve 19.a2, which verifies
$$
N_E = 19, \qquad E_\tors(\mathbb{Q}) = \mathcal{C}_1, \qquad E_\tors \big( \mathbb{Q}(\sqrt{-3}) \big) = \mathcal{C}_3.
$$

Therefore what we aim for here is a more complete understanding of the phenomenon, and give conditions under which we can be sure that $3|d$ implies $3|N_E$. This is still a work in progress but we have been able to prove results on this regard, like the following ones (appearing in Section 4 as Propositions \ref{prop:C1_to_C9} and \ref{caso N=3}):

\

\noindent\textbf{Proposition \ref{prop:C1_to_C9}.} \textit{Let $E/\mathbb{Q}$ be an elliptic curve such that $E_\tors(\mathbb{Q})$ is trivial and $E_\tors(K) =  \mathcal{C}_9$, with $K$ a quadratic field. Then, $E$ has bad reduction at 3.}

\

\noindent\textbf{Proposition \ref{caso N=3}.} \textit{Let $E/\mathbb{Q}$ be an elliptic curve and $K=\mathbb{Q}(\sqrt{d})$ be a quadratic extension such that there exists a point $P\in E[3]$ such that $P\in E(K)\setminus E(\mathbb{Q})$. Let us assume $3|d$ and $E$ has good reduction at 3 ($3 \nmid N_E$), then}
$$
E_\tors(K) =   E_\tors(\mathbb{Q})\times \mathcal{C}_3.
$$

\ 

\noindent \boxed{\textbf{The cases $\ell=5,7$.}} (Section 5). As in the $\ell=2$ case we have here:

\begin{theorem}\label{l57}
Let $E/\mathbb{Q}$ be an elliptic curve, $K$ a quadratic number field such that $E_\tors(\mathbb{Q}) \neq E_\tors(K)$. If $p=5,7$ ramifies in $K$, then $p \vert N_E$.
\end{theorem}

\

This result was already known and it can be proved from \cite[Theorem 1]{Olson} (using $d$--twists) and also from \cite[Theorem 1.5]{Melistas} (using reduction types in abelian varieties over more general base fields). Nevertheless we have applied the mod $\ell$ Galois representation techniques in order to (hopefully) find a new proof which would allow us to better understand the phenomenon. In particular, these two primes can share a common approach via the study of the inertia group of a prime $\Lambda\vert \ell$ of $\mathbb{Q}(E[\ell])$. Essentially, the size of the groups involved will render impossible the existence of points of order $\ell$ in $E_\tors(K)$, although it will need to be proved in a case--by--case argument, depending on the reduction type of $E$ at $\ell$.

Besides succeeding in finding an alternative proof, our arguments can be used to prove the following result, which cannot be derived from \cite{Olson,Melistas} and is, in fact, a stronger statement than Theorem \ref{l57}:

\begin{theorem}\label{addred}
   Let $E/\mathbb{Q}$ be an elliptic curve. Assume $K$ is a quadratic number field such that  there exists $P\in E(K)[\ell]\setminus E(\mathbb{Q})$, with $\ell\geq 3$ prime.  
    \begin{enumerate}
        \item For every prime $p\neq \ell$ that ramifies in $K$, $E$ has additive reduction at $p$, i.e. $p^2|N_E$.
        \item If $p=\ell>3$ ramifies in $K$, $E$ has additive reduction at $p$, i.e. $p^2|N_E$.
    \end{enumerate}
\end{theorem}

Section 6, finally, is devoted to conclusions and final remarks, including comments on future directions of our work. To make the reading of the paper a little easier, we have included an appendix where we collect some information on subgroups of $\mathrm{GL}_2(\mathbb{F}_5)$ and $\mathrm{GL}_2(\mathbb{F}_7)$ which are needed in Section 5.

\

The authors want to thank the anonymous referees for their insights which actually allowed us to realise Theorem \ref{addred} and which have meant an important contribution to the final form of this paper, often by shortening arguments. They also want to thank E. González--Jiménez and F. Najman for their conversations on the topic and support during the preparation of the paper. Finally, thanks are also due to M. Melistas who pointed us out his contributions on the matter.

\section{The prime $\ell=2$: The strict case}\label{sec:2_strict}

Let us suppose that $E/\mathbb{Q}$ is an elliptic curve and $K/\mathbb{Q}$ is a quadratic extension such that there exists a point $P\in E[2]$ with $P\in E(K)\setminus E(\mathbb{Q})$.

We then have the following diagram:
\begin{equation*}
 \xymatrix{\mathbb{Q}(E[2]) \ar@{-}[d] \ar@{-}@/^2pc/[dd]^{\overline{\rho}_{E,2}(G_{\mathbb{Q}})} \ar@{-}[d] \ar@{-}@/_2pc/[d]_{H}\\ K  \ar@{-}[d]^2 \\ \mathbb{Q}}
\end{equation*}

We will have that $\overline{\rho}_{E, 2}(G_{\mathbb{Q}})$ is a subgroup of $\mathrm{GL}_2(\mathbb{F}_2)$ which has a subgroup $H$ of index 2, which does have in turn a fixed point, different from $(0, 0)^t$, that is not a fixed point of $\overline{\rho}_{E, 2}(G_{\mathbb{Q}})$.

Note that $\mathrm{GL}_2(\mathbb{F}_2)$ has the following subgroups (up to conjugation) \cite{Zywina}:
$$
G_1 = \left\{ \mathrm{Id} \right\}, \qquad 
G_2 = \left\{ \mathrm{Id}, \begin{pmatrix} 1 & 1\\ 0 & 1\end{pmatrix} \right\}, \qquad
G_3 = \left\{ \mathrm{Id}, \begin{pmatrix} 1 & 1\\ 1 & 0\end{pmatrix}, \begin{pmatrix} 0 & 1\\ 1 & 1\end{pmatrix} \right\} 
$$
$$
G_4 = \mathrm{GL}_2(\mathbb{F}_2) = \left\{ \mathrm{Id}, \begin{pmatrix} 1 & 1\\ 0 & 1\end{pmatrix} \begin{pmatrix} 1 & 1\\ 1 & 0\end{pmatrix}, \begin{pmatrix} 0 & 1\\ 1 & 1\end{pmatrix}, \begin{pmatrix} 1 & 0\\ 1 & 1\end{pmatrix}, \begin{pmatrix} 0 & 1\\ 1 & 0\end{pmatrix} \right\}.
$$

\ 

Let us examine each case one by one:

\begin{enumerate}
\item $\overline{\rho}_{E, 2}(G_{\mathbb{Q}})=G_1$: The group $G_1$ has no subgroups of index 2.
\item $\overline{\rho}_{E, 2}(G_{\mathbb{Q}})=G_2$: We already have a point of $2$-torsion in $E(\mathbb{Q})$ (the point corresponding to $(1, 0)^t$), and in $E(K)$, all of the $2$-torsion is present. This case can occur (see examples).
\item $\overline{\rho}_{E, 2}(G_{\mathbb{Q}})=G_3$: In this case, $\mathbb{Q}(E[2])/\mathbb{Q}$ has order $3$, and therefore cannot contain a quadratic subextension.
\item $\overline{\rho}_{E, 2}(G_{\mathbb{Q}})=G_4$: It has a subgroup of index $2$, specifically $G_3$. But since $G_3$ has no fixed points, no torsion is added.
\end{enumerate}

\begin{example}
Let us see examples of curves with good reduction at $2$, but in such a way that $2$-torsion is added over a field that ramifies at the prime $2$:

\begin{itemize}
\item Curve $15.a3$: This curve has Galois group $\Gal(\mathbb{Q}(E[2])/\mathbb{Q})\simeq G_2$, and $E_\tors(\mathbb{Q})\simeq \mathcal{C}_2$. Over the field $\mathbb{Q}(\sqrt{-5})$, it obtains torsion $\mathcal{C}_2\times \mathcal{C}_2$.

\item Curve $17.a3$: This curve has Galois group $\Gal(\mathbb{Q}(E[2])/\mathbb{Q})\simeq G_2$, and $E_\tors(\mathbb{Q})\simeq \mathcal{C}_4$. Over the field $\mathbb{Q}(i)$, it obtains torsion $\mathcal{C}_2\times \mathcal{C}_4$.

\item Curve $15.a8$: This curve has Galois group $\Gal(\mathbb{Q}(E[2])/\mathbb{Q})\simeq G_2$, and $E_\tors(\mathbb{Q})\simeq \mathcal{C}_8$. Over the field $\mathbb{Q}(i)$, it obtains torsion $\mathcal{C}_2\times \mathcal{C}_8$.
\end{itemize}
    \end{example}

We must then ask ourselves now, in order to explore our conjecture, if it is possible for the torsion of an elliptic curve, with good reduction at $2$, to grow in such a way over a quadratic field $\mathbb{Q}(\sqrt{d})$, where $d$ is square-free and $2\vert d$. 

\

The rest of the section is devoted to the proof of the following:

\begin{proposition}\label{2strict}
    Let $E/\mathbb{Q}$ be an elliptic curve and $K=\mathbb{Q}(\sqrt{d})$ a quadratic extension. Suppose that there exists a point $P\in E(K)[2]\setminus E(\mathbb{Q})[2]$. Then, if $2\vert d$, $E$ has bad reduction at $2$.
\end{proposition}

\begin{proof}
Consider a minimal model for $E$ at $2$:
\begin{equation}\label{eq:minimalmodel}
y^2 + a_1xy + a_3y=x^3 + a_2x^2 + a_4x + a_6,
\end{equation}
and let us assume $E$ has good reduction at the prime $2$. Therefore the discriminant of the elliptic curve,
$$
\Delta= -b_2^2b_8  - 8b_4^3-27b_6^2 + 9b_2b_4b_6,
$$
is an odd number. The expressions for the integers $b_i$ in terms of the integers $a_i$ can be found, for example, in \cite[Chapter III]{Silverman}. 

Now, to calculate the $2$-torsion points, we use the $2$-division polynomial $\psi_2 = 2y + a_1x + a_3$; substituting 
$$
y \longmapsto \frac{1}{2}(-a_1x-a_3)
$$ 
into equation \eqref{eq:minimalmodel}. We obtain the equation with integral coefficients
\begin{equation}\label{eq:integercoeff}
0 = 4x^3 + b_2x^2 + 2b_4x + b_6.
\end{equation}

The $x$--coordinates of the nontrivial $2$-torsion points are the three roots of this polynomial, say $\alpha$, $\beta$, and $\gamma$, and we have $\mathbb{Q}(E[2])=\mathbb{Q}(\alpha, \beta, \gamma)$.

If there is already a $2$-torsion point over $\mathbb{Q}$, and we know that over a quadratic extension $K/\mathbb{Q}$ there is another torsion point, then we have precisely one single rational $2$-torsion point. This means that the above polynomial factors as
$$
4x^3 + b_2x^2 + 2b_4x + b_6 = 4(x-\alpha)(x-\beta)(x-c),
$$
where $c\in \mathbb{Q}$ and $\alpha, \beta$ are conjugate elements in some quadratic field $\mathbb{Q}(\sqrt{d})$, say
$$
\alpha = a + b\sqrt{d}, \qquad \beta = a-b\sqrt{d}.
$$

We are assuming $\mathbb{Q}(\sqrt{d})$ with $2|d$, therefore the ring of integers of $\mathbb{Q}(\sqrt{d})$ is $\mathbb{Z}[\sqrt{d}]$. Denote by $v_2$ the $2$-adic valuation in $\mathbb{Z}[\sqrt{d}]$, normalized such that $v_2(2)=1$, and take 
$$
v=v_2(\alpha)=v_2(\beta), \quad w=v_2(c).
$$

The discriminant of the polynomial $4x^3 + b_2x^2 + 2b_4x + b_6$ is 
$$
16 \Delta = 4^4(\alpha- \beta)^2(\alpha - c)^2(\beta- c)^2,
$$
and therefore,
$$
\Delta=4^2(\alpha- \beta)^2(\alpha - c)^2(\beta- c)^2,
$$
which must be an odd number, as we have good reduction at $2$.

\ 

Let us show now that $v, w \geq -2$, using the fact that Equation (\ref{eq:integercoeff}) has integral coefficients. 

\ 

First, if $v=w$, then the independent term of the polyomial $4(x-\alpha)(x-\beta)(x-c)$ has $2$-adic valuation $v_2(4\alpha\beta c)=2 + v + v + v\geq 0$, thus $v\geq -2/3$. In particular, $v\geq -2$. 

\ 

Assume now that $v\not=w$. If we look at the coefficient of the term of degree 1 of this polynomial, $4(c\alpha + c\beta + \alpha\beta)$, we obtain that its $2$-adic valuation is 
$$
2 + \min\{ v + w,  v + v_2(c + \alpha) \}= 2 + \min \{v + w, v + \min\{v, w\}\}= 2 + v + \min\{v, w\},
$$
and this must be greater than or equal to zero.

\ 

If $v<w$, we obtain that $2 + 2v\geq 0$, from which $v\geq -1$ and $w>v\geq -1$. 

\ 

Finally, if $v>w$, then we have that $v_2(\alpha + \beta)\geq \min\{v, v\}=v >w$, so that $v_2(\alpha + \beta)\not=w$. Therefore, the coefficient of the term of degree 2, $4(-\alpha - \beta- c)$, has $2$-adic valuation $2 + \min\{v_2(\alpha + \beta), w\}=2+w$. This number must be greater than or equal to zero, which yields that $w\geq -2$ and  $v>w\geq -2$.

\ 

Using again that the polynomial in Equation (\ref{eq:integercoeff}) lies in $\mathbb{Z}[X]$, we  have $-4\alpha\beta c \in \mathbb{Z}$. Taking $2$-adic valuation, we obtain
\begin{equation}\label{eq:producto}
2 + 2v + w \geq 0.
\end{equation}

Moreover, the condition that the discriminant $\Delta$ is odd leads us to the equation:
\begin{equation}\label{eq:valuationDelta}\begin{aligned}
0 =v_2(\Delta)&=v_2 \Big( 4^2(\alpha- \beta)^2(\alpha - c)^2(\beta- c)^2 \Big)\\ 
& = 4 + v_2(4b^2d) + v_2 \Big( \big( (a-c) + b\sqrt{d} \big)^2 \big( (a-c)- b\sqrt{d} \big)^2 \Big)\\ 
& = 6 + v_2 \big( b^2d \big) + v_2 \Big( (a-c)^2 -b^2 d)^2 \Big)\\ 
&=6 + v_2 \big( b^2d \big) + 2v_2 \Big( a^2 - b^2d - 2ac + c^2 \Big).\\
\end{aligned}
\end{equation}
  As $b\in \mathbb{Q}$, this equation implies that $-1 =-v_2(d) = 6+ 2v_2(b) + 2v_2 ( a^2 - b^2d - 2ac + c^2 )$ is even, a contradiction.
\end{proof}

\section{The prime $\ell=2$: The mixed case}\label{sec:2_mixed}

Let us continue with the case where a new torsion point $P\in E(K)$ appears, such that a nontrivial multiple $mP$ belongs to $E(\mathbb{Q})$. As mentioned in the Introduction, from Theorem \ref{GT}, we can reduce ourselves to the following hypothesis:
\begin{align*}
    \text{There exists $P\in E(K)[N]\setminus E(\mathbb{Q})$ such that $2P\in E(\mathbb{Q})$ and $o(P) = N$},
\end{align*}
where $N=4,8$ or $16$. We are going to study each case now. We will need to perform several changes of variables, so we need a notation for them. Given a curve with an equation in variables $x$ and $y$, we will call the variables of the new curve $x'$ and $y'$. After making the change, we will revert to calling the variables $x$ and $y$ for simplicity.

\subsection*{The case $N=4$}\label{caso N=4} This subsection is devoted to the proof of the following result:

\begin{proposition}\label{N=4}
    Let $E/\mathbb{Q}$ be an elliptic curve and $K = \mathbb{Q}(\sqrt{d})$ a quadratic extension. Let us suppose that there exists $P\in E(K)[4]\setminus E(\mathbb{Q})$ such that $\mathcal{O} \neq 2P\in E(\mathbb{Q})$. Then if $2|d$, $E$ has bad reduction at $2$.
\end{proposition}

\begin{proof}
Let us consider a minimal Weierstrass equation for $E$ at $2$ with integer coefficients 
$$
E: y^2 + a_1xy + a_3y=x^3 + a_2x^2 + a_4x + a_6.
$$ 
It can be assumed that $a_1,a_3\in \{0,1\}$ and $a_2\in\{-1,0,1\}$. Let us
assume $K = \mathbb{Q}(\sqrt{d})$ with $d$ square free and $2|d$. We proceed by contradiction, so assume that $E$ has good reduction at $2$. So, if we call $\Delta$ the discriminant of the minimal form of $E$, we know $\Delta$ is an odd integer. 

Let us consider the following change:
$$
x = x',\quad y = y' -\frac{1}{2}\left( a_1x' + a_3 \right).
$$
It leads us to the equation:
$$
 y^2= x^3 + \frac{b_2}{4}x^2 + \frac{b_4}{2}x + \frac{b_6}{4}.
$$
where formulas for $b_i$ can be found in \cite[Chapter III]{Silverman}. It is easy to see that the discriminant of this equation is still $\Delta$. In order to get an equation with integer coefficients, we make the change
$$
x' = 4x,\quad y' = 8y
$$
and we get the equation (in $\mathbb{Z}[x,y]$)
$$
y^2 = x^3 + b_2x^2 + 8b_4x + 16 b_6.
$$
Let us call $\Delta'$ the discriminant of the previous equation. We have:
$$
2^{12}\Delta = \Delta'.
$$
By hypothesis, $\mathcal{O} \neq2P\in E_\tors(\mathbb{Q})$. So there exists $2P=Q = (\gamma,0)$ with $\gamma\in \mathbb{Z}$ by Nagell-Lutz theorem. Now, we apply the change of variables
$$
y'=y, \quad x' = x - \gamma.
$$

In this way, we can assume that the point $Q=2P\in E(\mathbb{Q})$ is $Q=(0, 0)$. Moreover, the new curve has the form
$$
y^2 = x^3 + Ax^2 + Bx
$$
with
\begin{eqnarray*}
A &=& 3\gamma + b_2,\\
B &=& 3\gamma^2 + 2\gamma b_2 + 8b_4
\end{eqnarray*}
where we have the following equation between discriminants
\begin{equation}\label{eq:eq_between_discriminants}
2^{12}\Delta = 16(A^2B^2-4B^3).
\end{equation}

$E_\tors(\mathbb{Q})[2]$ contains a copy of $\mathcal{C}_2$, so
$$
E_\tors(\mathbb{Q})[2] =  \mathcal{C}_2 \quad \text{or} \quad E_\tors(\mathbb{Q})[2] =  \mathcal{C}_2 \times \mathcal{C}_2, 
$$
and we are going to get a contradiction in both cases. 

\ 

\noindent \fbox{\textbf{Case I:} $E_\tors(\mathbb{Q})[2] =  \mathcal{C}_2$.} We can use the following lemma proven in \cite[Lemma 13]{GonTorII}.

\begin{lemma}\label{lema C2}
    Let
    $$
    y^2 = x(x^2 +Ax +B)
    $$
    be an elliptic curve over $\mathbb{Q}$ with $E_\tors(\mathbb{Q})=  \mathcal{C}_2$. Then, there exists a quadratic field $K$ with $\mathcal{C}_4\leq E_\tors(K)$ if and only if $B = s^2$ for some $s\in \mathbb{Q}$. Moreover, in this situation $K$ is one of the following two fields (they might be the same):
    $$
    K_{\pm} = \mathbb{Q}\left( \sqrt{A\pm 2s}\right).
    $$
\end{lemma}
If we apply this lemma to our problem, we obtain $s\in \mathbb{Z}$ such that $s^2=B$ and $\mathbb{Q}(\sqrt{d})=\mathbb{Q}(\sqrt{A\pm 2s})$. So Equation \eqref{eq:eq_between_discriminants} becomes
$$
2^{12}\Delta = 16s^4(A^2-4s^2) = 16s^4(A-2s)(A + 2s).
$$
Applying the 2-adic valuation $v_2$ we get
\begin{equation}\label{eq:eq_between_discriminants_particular_case}
8 = 4v_2(s) + v_2(A-2s) + v_2(A+2s),
\end{equation}
 At least one among $v_2(A-2s), v_2(A+2s)$ is odd as $v_2(d)$ is odd. Hence, they must be both odd and distinct, as if they were the same, their sum would not be divisible by 4. For $v_2(A-2s)$ and $v_2(A+2s)$ to be distinct, we must have $v_2(A)=v_2(2s)$. Furthermore, $v_2(A-2s)$ and $v_2(A+2s)$ are two odd integer greater than $v_2(2s)$, so their sum is at least $2v_2(s) + 4 = 2v_2(s) + 6$, which implies $v_2(s)= 0$, and furthermore $v_2(A) = 1$.
\\~\\
Writing $A=2r$ with $r$ odd, we have
$$
8 = 2+ v_2(r-s) + v_2(r+s),
$$
with $v_2(r-s)$ and $v_2(r+s)$ even and positive. However, we cannot have both $r-s$ and $r+s$ divisible by $4$, since $r$ and $s$ are odd. Thus, at least one among $v_2(r-s),v_2(r+s)$ equals $1$, a contradiction.

\ 

\noindent \fbox{\textbf{Case II: } $E_\tors(\mathbb{Q})[2] = \mathcal{C}_2\times\mathcal{C}_2$.} For this purpose, we will use the following classical result in the literature of elliptic curves \cite[Theorem 4.2]{Knapp}.

\begin{lemma}\label{lema del Knapp}
    Let $E$ be an elliptic curve defined over a number field $L$ given by
    $$
    y^2 = (x-\alpha)(x-\beta)(x-\gamma) 
    $$
    with $\alpha,\beta,\gamma\in L$. For $P=(x_0,y_0)\in E(L)$, there exists $Q\in E(L)$ such that $2Q = P$ if and only if $x_0-\alpha,x_0-\beta,x_0-\gamma\in \left(L^*\right)^2$.
\end{lemma}

In our case, the curve is
$$
y^2 = x(x^2+ AX+ B) = x(x-\alpha)(x-\beta)
$$
with $\alpha,\beta\in \mathbb{Q}$ and we know that there exists $Q\in E(\mathbb{Q}(\sqrt{d}))\setminus E(\mathbb{Q})$ such that $2Q = (0,0)$. Thus 
$$
-\alpha,-\beta\in \left(\mathbb{Q}(\sqrt{d})^*\right)^2.
$$

Note that $-\alpha$ and $-\beta$ cannot be both squares in $\mathbb{Q}$. Should that be the case, let us write
$$
E(\mathbb{Q})[2] = \{0, P = (0,0), P_1 = (\alpha,0), P_2=(\beta,0)\} \subset E(\mathbb{Q}) [4].
$$

Then from the previous lemma, as $-\alpha$ and $-\beta$ a rational squares, there exists a point $R \in E(\mathbb{Q})[4]$ and the four points of $E[4]$ that map to $(0, 0)$ under multiplication by $2$ are precisely $R, R + (0, 0), R + P_1, R + P_2$. These are all rational, which contradicts our hypothesis that there is some point $Q$ of $4$-torsion which is not rational and such that $2Q=(0, 0)$. 

So this amounts to the existence of $a,b\in \mathbb{Z}$ such that one of the following mutually exclusive pairs of equalities holds:
$$
\{-\alpha=a^2d, -\beta=b^2\},\quad \{-\alpha = a^2, -\beta = b^2d\} \quad \text{or} \quad \{-\alpha = a^2d, -\beta=b^2d\}.
$$
Again, we will assume each of these cases and we will get a contradiction.

\

\noindent \fbox{II.A: $-\alpha = a^2d, -\beta=b^2$.} Equation \eqref{eq:eq_between_discriminants_particular_case} becomes
$$
\begin{array}{rclcl}
2^{12}\Delta &=& 16B^2(A^2-4B) &=& 16(\alpha\beta)^2((\alpha+\beta)^2-4\alpha\beta) \\ &=& 16(\alpha\beta)^2(-\alpha + \beta)^2 &=& 16a^4b^4d^2(a^2d -b^2)^2 
\end{array}
$$
Taking the $2$-adic valuation we get
$$
6 = 4v_2(a) + 4v_2(b) + 2v_2(a^2d-b^2), 
$$
note that by hypothesis $v_2(d)= 1$. The equation implies that $v_2(a)\in \{0,1\}$. If $v_2(a) = 0$, we get
$$
6 = 4v_2(b) + 2v_2(a^2d-b^2).
$$
Again, $v_2(b)\in \{0,1\}$. If $v_2(b) = 0$, we get a contradiction easily. So $v_2(b) = 1$ and we have the following relations
\begin{eqnarray*}
    A &=& a^2d + b^2 = 3\gamma + b_2,\\
    B &=& a^2b^2d = 3\gamma^2+ 2\gamma b_2 + 8b_4.
\end{eqnarray*}
Taking valuations, we get
\begin{eqnarray*}
    1 &=& v_2(3\gamma + b_2),\\
    3 &=& v_2(3\gamma^2+ 2\gamma b_2 + 8b_4).
\end{eqnarray*}
From the second equation we deduce that $v_2(\gamma) \geq 1$. Now using the first equation, $b_2$ must be even, but we know that $b_2 = a_1^2 + 4a_2$ with $a_1\in\{0,1\}$. Therefore, $b_2 = 4a_2\in\{0,4,-4\}$ and in each of these cases $v_2(\gamma) = 1$, because $1 = v_2(3\gamma + b_2)$. If $b_2 =0$, we obtain
$$
3 = v_2(B) = v_2(3\gamma^2 + 8b_4),
$$
which implies that $1 = v_2(\gamma)\geq 2$, a contradiction. Therefore, $v_2(b_2) = 2$ and we have the relation 
$$
3 = v_2(3\gamma^2 + 2\gamma b_2 + 8b_4),
$$
which implies $1 = v_2(\gamma)\geq 2$, a contradiction.

\    

Assume now $v_2(a) = 1$, we get
$$
2 = 4v_2(b) + 2v_2(a^2d -b^2).
$$
So $v_2(b) = 0$ and $1= v_2(a^2d -b^2) = 0$, a contradiction.

\ 
    
\noindent \fbox{II.B: $-\alpha = a^2, -\beta=b^2d$.} This case is analogous to the previous one.

\ 
    
\noindent \fbox{II.C: $-\alpha = a^2d, -\beta=b^2d$.} Now we have the following equations:
\begin{eqnarray*}
    A &=& d(a^2 + b^2),\\
    B &=& a^2b^2d^2,\\
    2^{12}\Delta &=&16 a^4b^4d^6(a^2-b^2)^2.
\end{eqnarray*}
Taking the $2$-adic valuation in the last equation, we get
$$
2 = 4v_2(a) + 4v_2(b) + 2v_2(a-b) + 2v_2(a+b).
$$
So $v_2(a)=v_2(b) = 0$ and
$$
2 = 2v_2(a-b) + 2v_2(a+b) \geq 4,
$$
a contradiction.

\

This concludes the proof of Proposition \ref{N=4}. 
\end{proof}

\begin{remark}
It is a natural question if this statement still holds when we change the hypothesis of $2|d$ by the weaker hypothesis that $2$ ramifies in $K$. The answer is that it does not hold, a counterexample is the curve $17.a2$.
\end{remark}

\subsection*{The cases $N=8,16$}\label{subsec:N=8y16} 

\noindent Now we address the case where we have an elliptic curve $E/\mathbb{Q}$  with a point $P\in E(K)\setminus E(\mathbb{Q})$ of order $8$ or $16$ for a quadratic extension $K=\mathbb{Q}(\sqrt{d})$, such that $2P\in E(\mathbb{Q})$. 

\ 

First, some remarks which are common for both cases. Consider  a minimal Weierstrass equation at $2$:
$$
E_1: y^2 + a_1xy + a_3y = x^3 + a_2x^2 + a_4x + a_6,
$$
with $a_1,\dots,a_6\in \mathbb{Z}$. In order to understand the properties of $P$ in terms of $K = \mathbb{Q}(\sqrt{d})$, we are going to take our initial Weierstrass form into a more suitable form (the so--called Tate normal form). 

\ 

The first step is taking the minimal Weierstrass form into the already known form
\begin{equation}\label{eq_EllipticCurveE[2]}
E_2:y^2 = x^3 + b_2x^2 + 8b_4x + 16b_6,
\end{equation}
where we have the following relation between the respective discriminants of $E_1$ and $E_2$
$$
2^{12} \Delta_1 = \Delta_2,
$$
and where, by the Nagell-Lutz theorem, the point $Q=2P= (x_1,y_1)\in E_2(\mathbb{Q})$ has integer coordinates. 

\ 

A Tate normal form is an equation of the form
$$
y^2 + (1-c)xy - by = x^3 - bx^2.
$$

We are going to transform the equation of $E_2$ into one of these and the point $Q$ will be sent to $(0,0)$. First, we are going to get an equation of the form
$$
E_3: y^2 + \overline{a}_1xy + \overline{a}_3y = x^3 + \overline{a}_2x^2.
$$
Every change of variables that preserves the Weierstrass form must be like
$$
x = u^2x' + r,\quad y = u^3y' + su^2x' + t.
$$
The only parameter that changes the discriminant is $u$, so we impose $u=1$. Besides, we want to send $Q$ to $(0,0)$. Since 
$$
(x_1,y_1) \longmapsto \big( x_1-r, y_1-s(x_1-r) -t \big),
$$ 
we require that $r= x_1$ and $t=y_1$.  This makes $0$ the independent term of the equation. Now, we choose $s$ in order to force the coefficient in $x$ to be $0$, which give us the following relation
$$
0 = 8b_4 + 2b_2x_1 + 3x_1^2 -2sy_1.
$$
So, we have
\begin{equation}\label{eq:s}
s = \frac{8b_4 + 2b_2x_1 + 3x_1^2}{2y_1}.
\end{equation}

Note that $y_1\neq 0$ because $2Q \neq \mathcal{O}$. We obtain the following form
$$
E_3: y^2 + \overline{a}_1xy+ \overline{a}_3y = x^3 + \overline{a}_2x^2,
$$
with
$$
\def\arraystretch{1.5}
\left\{ \begin{array}{lcl} 
\overline{a}_1 & = & 2s,\\ 
\overline{a}_2 & = & -s^2+ 3x_1 +b_2, \\ 
\overline{a}_3 & = & 2y_1, \\ 
\Delta_3 & = & 2^{12} \Delta_1.
\end{array} \right.
$$

Note that $\overline{a}_2=0$ if and only if the point $Q$ is a $3$-torsion point (see \cite[V.5]{Knapp}).
In order to get a Tate normal form, we just have to equalize the coefficients of $y$ and $x^2$. We get the equation
$$
E_4: y^2 + \Tilde{a}_1xy+ \Tilde{a}_3y = x^3 + \Tilde{a}_2x^2,
$$
through the change of variables
$$
x \longmapsto \left(\frac{\overline{a}_3}{\overline{a}_2}\right)^2 x, \qquad y \longmapsto \left(\frac{\overline{a}_3}{\overline{a}_2}\right)^3 y.
$$

In this way, we have (mind $\overline{a}_3 = 2y_1 \neq 0$)
$$
\Tilde{a}_1 = \dfrac{\overline{a}_1\overline{a}_2}{\overline{a}_3}, \qquad 
\Tilde{a}_2 = \Tilde{a}_3 = \dfrac{\overline{a}_2^3}{\overline{a}_3^2}.
$$

\

If we call $b = -\Tilde{a}_2= -\Tilde{a}_3$ and $c= 1- \Tilde{a}_1$, we get the usual Tate normal form
$$
\mathcal{T}_{b,c}:  y^2+ (1-c)xy - by = x^3 - bx^2
$$
with the following relation between discriminants
\begin{equation}\label{eq:relation_discriminants}
\Delta_{b,c} = 2^{12}\left(\frac{\overline{a}_2}{\overline{a}_3}\right)^{12} \Delta_1.
\end{equation}

We can prove now the following result, which deals with the case $N=8$.

\begin{proposition}\label{N=8}
    Let $E/\mathbb{Q}$ be an elliptic curve and $K=\mathbb{Q}(\sqrt{d})$ with $d$ a squarefree even integer. If there exists a point $P\in E(K)[8]$ of order 8 such that $2P\in E(\mathbb{Q})$ then $E$ has bad reduction at $2$.
\end{proposition}

\begin{proof}
    We go as usual by contradiction, assuming that $E$ has good reduction at $2$, i.e. $\Delta_1$ is an odd integer. In order to get a Tate normal form, we use the previous coordinates changes. Since $Q=2P$ has order $4$, $c$ must be $0$ (see \cite[Chapter 4, $\S$4]{Husemoller}). So our Tate normal form is
     \begin{equation}\label{eq:Tate_with_c=0}
    \mathcal{T}_{b,0}: y^2 + xy -by = x^3 - bx^2,
    \end{equation}
    and, furthermore, we must have $1 = \Tilde{a}_1$ and hence $\overline{a}_1\overline{a}_2 = \overline{a}_3$.

    \
    
    We have that ${\mathcal C}_4\subset E_\tors(\mathbb{Q})$ and that, over a quadratic field $K$, ${\mathcal C}_8\subset E_\tors(K)$. First of all, we check which rational torsion subgroups contain ${\mathcal C}_4$, and we obtain that $E_\tors(\mathbb{Q})\in\{{\mathcal C}_4,{\mathcal C}_8,{\mathcal C}_{12},{\mathcal C}_2\times {\mathcal C}_4,{\mathcal C}_2\times {\mathcal C}_8\}$. Let us look at each case:
    \begin{enumerate}
        \item $E_\tors(\mathbb{Q}) = {\mathcal C}_4$. This case can occur.
        \item $E_\tors(\mathbb{Q}) = {\mathcal C}_8$. The only growths that can occur (see Theorem \ref{GT}) are $E_\tors(K) = {\mathcal C}_{16}$ or $E_\tors(K) = {\mathcal C}_{2}\times {\mathcal C}_{8}$. In the first case, there are no new points of order $8$, but in the second case we obtain points of order $8$ which are defined only after a base change to $K$. This case is considered in \cite[Section 4.1]{GonTorII}; the only possibility for $K$
        is $K=\mathbb{Q}(\sqrt{\Delta_E})$, which contradicts the condition $2|d$.
        \item $E_\tors(\mathbb{Q}) = {\mathcal C}_{12}$. According to Theorem \ref{GT}, the only possible growth is to ${\mathcal C}_2\times {\mathcal C}_{12}$. The same reasoning as above shows that this cannot happen.
        \item $E_\tors(\mathbb{Q}) = {\mathcal C}_2\times {\mathcal C}_4$. The only possible growth that allows for a point of order $8$ is to the group ${\mathcal C}_2\times {\mathcal C}_8$. This case can occur.
        \item $E_\tors(\mathbb{Q}) = {\mathcal C}_2\times {\mathcal C}_8$. In this case no growth can occur over a quadratic field.
    \end{enumerate}
    Thus, we have two possibilities: either $E_\tors(\mathbb{Q}) = {\mathcal C}_4$ or $E_\tors(\mathbb{Q}) = {\mathcal C}_2\times {\mathcal C}_4$. In both cases, we are going to apply \cite[Lemma 14]{GonTorII}. 
    
\begin{lemma}
Let $E$ be an elliptic curve defined over $\mathbb{Q}$ with $E_\tors(\mathbb{Q})= \mathcal{C}_4$. Let $t\in\mathbb{Q}$ such that $E$ is $\mathbb{Q}$--isomorphic to $\mathcal T_{t,0}$. There exists a quadratic field $K$ with $E_\tors(K) = \mathcal{C}_8$ if and only if $t=-s^2$ for some $s\in\mathbb{Q}$. 

Moreover, $K$ must be of the form $K_{\pm} = \mathbb{Q} (\sqrt{1\pm 4 s})$ and, in this situation, $K_+ \neq K_-$.
\end{lemma} 
        
    In this result the hypothesis states $E_\tors(\mathbb{Q}) = {\mathcal C}_4$. However, in \cite[Section 4.2]{GonTorII}, the authors explain that the first part of the proof of this lemma is valid also in the case when $E_\tors(\mathbb{Q}) = {\mathcal C}_2\times {\mathcal C}_4$ which is the part we actually need.
    
    Then we have that there exist $r\in \mathbb{Q}$ such that $b = -r^2$ and $K=\mathbb{Q}(\sqrt{1\pm 4r})$. Let us write $r = p/q$ with $\gcd(p,q) = 1$. So $\mathbb{Q}(\sqrt{d}) = \mathbb{Q}(\sqrt{q(q\pm 4p)})$, which implies that $2|q(q\pm 4p)$. Hence $2|q$ and $2\nmid p$. We will look closely first at the precise value of $v_2(q)$.
    
    \ 
    
    The discriminant of Equation \eqref{eq:Tate_with_c=0} is $b^4(1+16b) = r^8(1-16r^2)$, so from Equation \eqref{eq:relation_discriminants} we obtain
    $$
    \Delta_1 2^{12}= r^8(1+4r)(1-4r)\overline{a}_1^{12}. 
    $$
    Clearing denominators we get 
    $$
    q^{10}\Delta_1 2^{12}y_1^{12}= p^8(q+4p)(q-4p)(8b_4 + 2b_2x_1+ 3x_1^2)^{12}.
    $$
    
    Taking the $2$-adic valuation we get
    \begin{equation}\label{eq:valuation_discriminant}
    10v_2(q) + 12 + 12v_2(y_1) =  v_2(q+4p) + v_2(q-4p) + 12v_2(8b_4 + 2b_2x_1+ 3x_1^2).
    \end{equation}
    
    If $v_2(q) = 1$, it is equivalent to 
    $$
    22 + 12 v_2(y_1)= 2 + 12v_2(8b_4 + 2b_2x_1+ 3x_1^2),
    $$
    which leads to $5 = 3\left(v_2(8b_4 + 2b_2x_1+ 3x_1^2)-v_2(y_1)\right)$, a contradiction.
    
    So we must have $v_2(q)\geq 2$. From the definition of $b$ we get that 
    \begin{equation}\label{eq:b}
    -\frac{p^2}{q^2} = b = -\frac{\overline{a}_2^3}{\overline{a}_3^2} = -\frac{\overline{a}_3}{\overline{a}_1^3},
    \end{equation}
    as $\overline{a}_1\overline{a}_2 = \overline{a}_3$ in this case. Taking the $2$-adic valuation and using the definition of $\overline{a}_i$, we get 
    $$
    3v_2(8b_4 + 2b_2x_1+ 3x_1^2) = 2v_2(q)+1+4v_2(y_1). 
    $$
    If we multiply the last equation by $4$ and we substitute on Equation \eqref{eq:valuation_discriminant}, we get
    \begin{equation}\label{eq:b2}
    10v_2(q) + 12 + 12v_2(y_1) = v_2(q+4p) + v_2(q-4p) + 8v_2(q)+4+16v_2(y_1),
    \end{equation}
    which is equivalent to
    \begin{equation}\label{eq:valuation_q}
    2v_2(q) + 8 = 4v_2(y_1) + v_2(q+4p) + v_2(q-4p).
    \end{equation}
    Now let us assume that $v_2(q)\geq 3$,  so $2v_2(q) + 8 = 4v_2(y_1) + 4$. Dividing by $2$, the equation (\ref{eq:valuation_q}) becomes
    $$
    v_2(q) = 2(v_2(y_1)-1),
    $$
    which implies that $v_2(q)$ is even. However there exist odd numbers $M,N$ such that $ q = 2^{v_2(q)}N$ and $q\pm 4p = 4M$. But this implies that
    $$
    \mathbb{Q} \left( \sqrt{d} \right)= \mathbb{Q} \left( \sqrt{NM} \right),
    $$
    which is a contradiction because $2|d$ and $d$ is square-free. 
    
    Therefore we can affirm $v_2(q) = 2$.  Substituting it in the equation \eqref{eq:b2}, it becomes
    $$
    4(3-v_2(y_1)) = v_2(q+4p)+ v_2(q-4p) \geq 6,
    $$
    where the last inequality comes from the fact that $v_2(q\pm 4p)\geq 3$, as we know that $v_2(p) = 0$ and $v_2(q) = 2$. The inequality implies that $v_2(y_1) \in\{0,1\}$,  so we can separate the proof in two cases.
    
    \ 
    
    \noindent \fbox{\textbf{Case I:} $v_2(y_1)=0$.}
    If $v_2(y_1) = 0$, taking the $2$-adic valuation in the equation \eqref{eq:b}, we get
    $$
    -4 = v_2(b) = 3v_2(\overline{a}_2) - 2v_2(\overline{a}_3) = 3v_2(\overline{a}_2) -2.
    $$
    Hence, $-2 = 3v_2(\overline{a}_2)$, a contradiction. 
    
     \ 
    
    \noindent \fbox{\textbf{Case II:} $v_2(y_1)=1$.}
    In this case, $v_2(\overline{a}_3)=1 + v_2(y_1)=2$. Now, from Equation \eqref{eq:b}, we obtain that $-2v_2(q)=3v_2(\overline{a}_2) - 2v_2(\overline{a}_3)$. But in our case $v_2(q)=2$, so we must have $v_2(\overline{a}_2) = 0$. As $\overline{a}_1\overline{a}_2 = \overline{a}_3$,
    $$
    0 = v_2(\overline{a}_1) + v_2(\overline{a}_2) - v_2(\overline{a}_3).
    $$
    Therefore, $v_2(8b_4 + 2b_2x_1 + 3x_1^2) = 3$, which implies that $v_2(x_1)>0$. Using that $(x_1,y_1)$ is a point on the curve and $v_2(y_1)=1$ we get
    $$
    2 = v_2(x_1^3+ b_2x_1^2 + 8b_4x_1 + 16b_6) = v_2(x_1^3 + b_2x_1^2) = 2v_2(x_1) + v_2(x_1 + b_2).
    $$
    So, $v_2(x_1) = 1$ and $v_2(b_2) = 0$.  The pair $(x_1,y_1)$ is a solution of the equation 
    $$
    y^2 = x^3 + b_2x^2+8b_4x + 16b_6.
    $$
    The corresponding point in the original equation, i.e.
    $$
    y^2 + a_1xy + a_3y = x^3+a_2x^2+a_4x+a_6,
    $$
    is 
    $$
    \Big(\dfrac{x_1}{4},\dfrac{y_1-a_1x_1-4a_3}{8}\Big),
    $$
    which is a $4$-torsion point. So, a strong version of Nagell-Lutz Theorem \cite[VIII.7, Theorem 7.1]{Silverman} implies that 
    $$
    -1=v_2\left(\dfrac{x_1}{4}\right)\geq 0,
    $$
    a contradiction.
\end{proof}

Note that we cannot change the hypothesis of $2|d$ by $2$ ramifies over $K$, because the curve $15.a4$ has good reduction in $2$, its torsion over $\mathbb{Q}$ is $\mathcal{C}_4$ and over $\mathbb{Q}(\sqrt{3})$ (where $2$ ramifies) is $\mathcal{C}_8$.

\

When $N = 16$, we do not actually need the hypothesis $2|d$.

\begin{proposition}\label{N=16}
    Let $E/\mathbb{Q}$ be an elliptic curve and $K/\mathbb{Q}$ a quadratic extension. If there exists $P\in E(K)[16]$ of order 16, then $E$ has bad reduction in $2$.
\end{proposition}

\begin{proof}
    Because of \cite[Theorem 1.1]{Bruin2017}, any point $P\in E(K)$ of order $16$, satisfy that $Q=2P\in E(\mathbb{Q})$. We go again by contradiction and we assume that $E$ has good reduction in $2$, i.e. $\Delta_1$ is an odd integer. In order to get a Tate normal form, we use again the previous coordinates changes. Since $Q$ is an order $8$ point, there exists $t\in \mathbb{Q}$ such that
    $$
    c = \dfrac{(2t-1)(t-1)}{t},\quad b = (2t-1)(t-1).
    $$
    (see \cite[Chapter 4, $\S$ 4]{Husemoller}). In this way, the discriminant of the Tate normal form
    $$
    \mathcal{T}_{b,c}:  y^2+ (1-c)xy - by = x^3 - bx^2
    $$
    is 
    $$
    \Delta_{b,c} = \dfrac{(1-2t)^4(t-1)^8(8(t-1)t+1)}{t^4}.
    $$
Note that, since we have that $2P\in E_\tors(\mathbb{Q})$ is a point of order eight, then $E_\tors(\mathbb{Q})={\mathcal C}_8$ or ${\mathcal C}_2\times {\mathcal C}_8$. In the second case, there is no room for growth over a quadratic extension (cf.~\cite[Theorem 2]{GonTorII}). So we have $E_\tors(\mathbb{Q})={\mathcal C}_8$.
    
    Now, \cite[Lemma 16]{GonTorII} tell us that there exists $r\in \mathbb{Q}$ such that $t = r^2/(r^2+1)$ and 
    $$
    K = \mathbb{Q} \left( \sqrt{(r^4-1)(r^2\pm2r-1)} \right).
    $$
    Let us write $r=p/q$ with $\gcd(p,q) = 1$. Now, it follows 
    $$
    K = \mathbb{Q} \left( \sqrt{(p^4-q^4)(p^2\pm 2pq- q^2)} \right).
    $$
    On the other hand, we compute
    $$
    t= \frac{r^2}{r^2+1} = \frac{p^2}{p^2+ q^2}, \quad t-1 = \frac{-q^2}{p^2+ q^2}
    $$
    and 
    $$
    8(t-1)t+1 = \frac{(p^2+q^2)^2-8(pq)^2}{(p^2+ q^2)^2},\quad 1-2t=\frac{q^2-p^2}{p^2+ q^2}
    $$
    to compute the discriminant
    $$
   \Delta_{b,c}=\frac{(q^2-p^2)^4q^{16}((p^2+q^2)^2-8(pq)^2)}{p^8(p^2+q^2)^{10}}.
    $$

    We will now substitute this expression into Equation \eqref{eq:relation_discriminants}. First we compute the expression of $(\overline{a}_2/\overline{a}_3)^{12}$ in terms of $p$, $q$ and $y_1$.
    We have that $\overline{a}_2^3=-b\overline{a}_3^2$, hence
    $$
    \left(\frac{\overline{a_2}}{\overline{a}_3}\right)^{12}=\frac{(\overline{a}_2^3)^4}{\overline{a}_3^{12}}=\frac{b^4\overline{a}_3^8}{\overline{a}_3^{12}}=\frac{b^4}{\overline{a}_3^4}=\frac{b^4}{(2y_1)^4}.
    $$
    Thus
    $$
2^{12}\Delta_1=\left(\frac{\overline{a}_3}{\overline{a}_2}\right)^{12}\Delta_{b, c}=\frac{(2y_1)^4}{b^4}\Delta_{b, c}
    $$
    Now, we can apply all the previous formulas expressing $b$ and $\Delta_{b, c}$ in terms of $p$ and $q$, and after clearing denominators we get    
    \begin{equation}\label{eq:Delta_p_q}
        2^8\Delta_1 p^8(p^2 + q^2)^2=q^8 \big( (p^2 + q^2)^2 - 8(pq)^2 \big)y_1^4. 
    \end{equation}
    
    We divide now the proof in three cases:

    \ 
    
    \noindent \fbox{\textbf{Case I: } $v_2(p)=v_2(q)=0$.} Taking $2$-adic valuation in Equation (\ref{eq:Delta_p_q}) we have 
    $$
    8 + 2v_2(p^2 + q^2) =v_2((p^2 + q^2)^2 - 8(pq)^2) + 4v_2(y_1)
    $$
    Note that the sum of the squares of two odd numbers is always congruent to $2$ modulo $4$, so it has valuation 1. Therefore the above equation reduces to
    $$
    8 + 2 =2 + 4v_2(y_1),
    $$ and therefore $v_2(y_1)=2$. Moreover, $v_2(\overline{a}_3)=v_2(2y_1)=3$. Let us note that 
    $$
    1- \dfrac{\overline{a}_1\overline{a}_2}{\overline{a}_3} = c = \dfrac{q^2(q^2-p^2)}{(p^2+q^2)p^2},
    $$
    which implies that
    $$
    -\dfrac{\overline{a}_1\overline{a}_2}{\overline{a}_3} = \dfrac{(q^4-p^4)-2(pq)^2}{p^2(p^2+q^2)}.
    $$
    As $v_2(q^4-p^4) \geq 3$, we have that 
    $$
    v_2(\overline{a}_1) + v_2(\overline{a}_2)-3 = v_2((q^4-p^4)-2(pq)^2)- v_2(p^2+q^2) = 1-1 = 0,
    $$
    which is equivalent to 
    $$
    v_2(\overline{a}_2)+ v_2(8b_4+ 2b_2x_1 + 3x_1^2) = 5.
    $$
    Similarly, we can use the definition of $b$ to get 
     \begin{equation}\label{eq:15}
    - \dfrac{\overline{a}_2^3}{\overline{a}_3^2} = b = \dfrac{q^2(q^2-p^2)}{(p^2+q^2)^2},
    \end{equation}
    Taking $2$-adic valuation in the equation above, we get
    \begin{equation}\label{eq:valoracion_a2}
    3v_2(\overline{a}_2) = v_2(q^2-p^2) +4.
    \end{equation}

    Since the point $(x_1, y_1)$ belongs to the elliptic curve $E_2$, defined by Equation (\ref{eq_EllipticCurveE[2]}), we can conclude that $8\vert x_1^2(x_1 + b_2)$. We have two possibilities:
    \begin{enumerate}
        \item[(a)] $x_1$ is odd. In this case, $8\vert (x_1 + b_2)$. This allows us to compute the $2$-adic valuation of $s=(8b_4 + 2b_2x_1 + 3x_1^2)/(2y_1)$, namely 
        $$
        v_2(8b_4 + 2b_2x_1 + 3x_1^2)=v_2(8b_4 + x_1(2b_2 + 2x_1) + x_1^2)=0,
        $$
        thus $v_2(s)=-3$. But this implies that 
        $$
        v_2(\overline{a}_2)=v_2(-s^2 + 3x_1 + b_2)=-6,
        $$
        which contradicts Equation \eqref{eq:valoracion_a2}.

        \item[(b)] $x_1$ is even. Note that $v_2(x_1)\geq 2$. Indeed, if $v_2(x_1) = 1$, Equation (\ref{eq_EllipticCurveE[2]}) implies 
        $$
        4 = v_2(x_1^3+ b_2x_1^2+8b_4x_1 + 16b_6),
        $$
        so 
        $$
        4\leq 2v_2(x_1) + v_2(x_1+ b_2) = 2 + v_2(x_1+b_2).
        $$
        
        As $b_2 = 4a_2 + a_1^2$, $a_1\in\{0,1\}$ and $a_2\in\{-1,1,0\}$, we get a contradiction easily. Now, we have $v_2(8b_4+2b_2x_1+3x_1^2)\geq 3$ which implies that
        $$
        0\leq v_2(-s^2 +3x_1 + b_2)=v_2(\overline{a}_2) = 5- v_2(8b_4+2b_2x_1+3x_1^2) \leq 2.
        $$
        In the first inequality we have used that $s = (8b_4+2b_2x_1+3x_1^2)/(2y_1)$. The last inequality must be consistent with Equation \eqref{eq:valoracion_a2}:
        $$
        3v_2(\overline{a}_2) = v_2(q^2-p^2) + 4.
        $$
        Then $v_2(\overline{a}_2) = 2$ and $v_2(q^2-p^2) = 2$. This implies that
        $$
        v_2(p\pm q) = 1.
        $$
        Therefore $p\pm q \equiv 2 \mod 4$. Then, adding both congruences we get
        $$
        2p \equiv 0 \mod 4,
        $$
        which implies that $v_2(p) = 1$, a contradiction.
        \end{enumerate}

    \
    
    \noindent \fbox{\textbf{Case II: } $v_2(p)\neq 0$, $v_2(q) = 0$.} Taking the $2$-adic valuation  in Equation \eqref{eq:Delta_p_q}, we get
    $$
    8  + 8v_2(p)= 4v_2(y_1), 
    $$
    which implies that $v_2(y_1) = 2(1+v_2(p))$.  In particular, it is an even number greater than or equal to $4$. Note that $(x_1, y_1)$ is a torsion point of the curve $E_2$, which has integer coefficients. With a change of variables that preserves the coordinate $y$ and the discriminant, we can transform the equation of $E_2$ into an equation of the form $y^2=x^3 + Ax + B$, and apply Nagell-Lutz locally at $2$ (see for example \cite[VII, Theorem 3.4]{Silverman} and the proof of \cite[VIII, Corollary 7.2]{Silverman}). 
    
    We conclude that $y_1^2|\Delta_2=2^{12}\Delta_1$. Therefore, $v_2(y_1)\in\{4,6\}$. In addition, Equation \eqref{eq:15} implies that
    $$
    \overline{a}_2^3=\frac{-q^2(q^2-p^2)}{(p^2+ q^2)^2}\overline{a}_3^2.
    $$
    Taking again the $2$-adic valuation, we get
    $$
    3v_2(\overline{a}_2)2v_2(\overline{a}_3) = 2v_2(2y_1)\in\{10, 14\},
    $$
    a contradiction.
    
    \ 
 
    \noindent \fbox{\textbf{Case III: } $v_2(p) = 0$, $v_2(q)\neq 0$.} Taking the $2$-adic valuation in Equation \eqref{eq:Delta_p_q} we get the following equation
    $$
    8  = 4v_2(y_1) + 8v_2(q).
    $$
    Since $v_2(q)\geq 1$, we obtain that $v_2(y_1)=0$. Since $(x_1, y_1)$ belongs to the elliptic curve $E_2$ with Equation \eqref{eq_EllipticCurveE[2]}, we conclude that $x_1$ must be odd. This enables us to compute the valuation of $s$ from Equation \eqref{eq:s} and conclude that $v_2(s)=-1$, and therefore $v_2(\overline{a}_2)=-2$. Furthermore 
    $$
    v_2(b)=v_2 \left( \frac{(q^2-p^2)q^2}{(p^2 + q^2)^2} \right) = 2v_2(q)\geq 0.
    $$ 
    Therefore 
    $$
    0\leq v_2(b)=3v_2(\overline{a}_2) - 2v_2(\overline{a}_3)=3v_2(\overline{a}_2) - 2v_2(2y_1)-6 - 2\cdot 1=-8,
    $$
    which is a contradiction.
\end{proof}

\begin{corollary}[Alternative statement of Proposition \ref{N=16}]\label{cor16}
There are no elliptic curves defined over a quadratic field $K$ with $\mathcal{C}_{16} \subset E_\tors(K)$  and good reduction at $2$.
\end{corollary}
\begin{proof}
    It is a consequence of the previous result and \cite[Theorem 1.1]{Bruin2017}
\end{proof}
Note that the results of this section, along with Proposition \ref{2strict}, prove Theorem \ref{l2}.

\section{The prime $\ell=3$}

Let $E/\mathbb{Q}$ be an elliptic curve and $K=\mathbb{Q}(\sqrt{d})$  be a quadratic extension of $\mathbb{Q}$ such that there exists a point $P\in E[3]$ satisfying that $P\in E(K)\setminus E(\mathbb{Q})$. As we discussed in the introduction, whenever $3$ divides $d$ but $3$ is a prime of good reduction for $E$, we can reduce to this case (the strict case).

Then we have that $\overline{\rho}_{E, 3}(G_{\mathbb{Q}})$ is a subgroup of $\mathrm{GL}_2(\mathbb{F}_3)$ with a subgroup $H$ of index 2 (namely $\overline{\rho}_{E, 3}(G_K)$) with a fixed point other than $(0, 0)^t$ that is not a fixed point of $\overline{\rho}_{E, 3}(G_{\mathbb{Q}})$. 

Consider the lattice of subgroups of $\mathrm{GL}_2(\mathbb{F}_3)$ up to conjugation (according to \cite{Yvon}; the notation for the subgroups is taken from \cite{Zywina}):

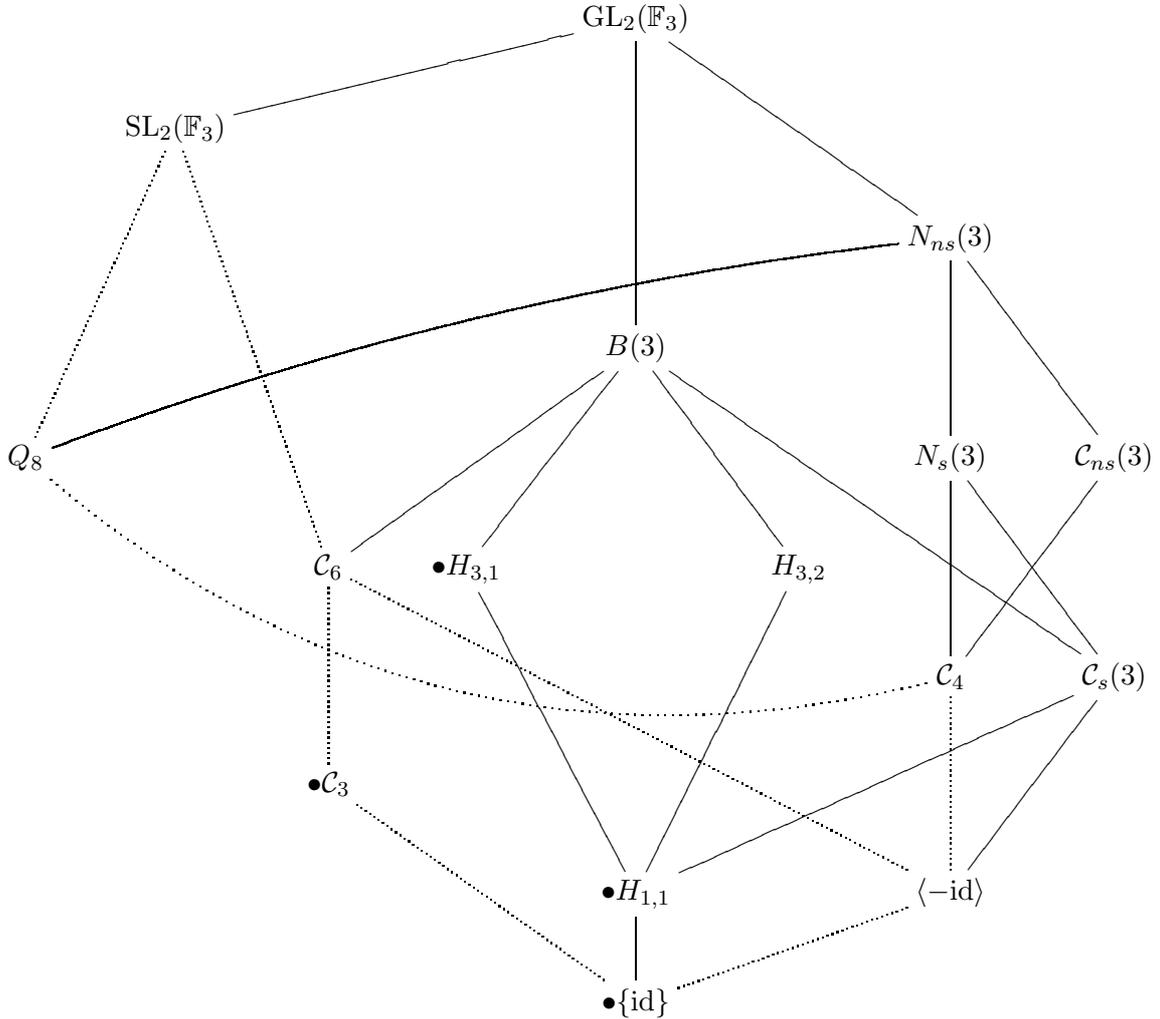
\begin{equation*}
 \xymatrix{\ &\ & \ & \ & \GL_2(\mathbb{F}_3) & \ & \ & \ \\
           \ & \SL_2(\mathbb{F}_3)\ar@{-}[urrr] & \ &\ & \ & \ & \ & \ \\
           \ &\ & \ & \ & \ & \ & N_{ns}(3) \ar@{-}[uull]& \ \\
           \ &\ & \ & \ & B(3) \ar@{-}[uuu]& \ & \ & \ \\
           Q_8 \ar@{.}[uuur] \ar@{-}@/^1pc/[uurrrrrr]&\ & \ & \ & \ & \ & N_s(3) \ar@{-}[uu] & {\mathcal C}_{ns}(3) \ar@{-}[uul]\\
           \ &\ & {\mathcal C}_6 \ar@{.}[uuuul] \ar@{-}[uurr]& \bullet H_{3, 1} \ar@{-}[uur] & \ & H_{3, 2} \ar@{-}[uul]& \ & \ \\
           \ &\ & \ & \ & \ & \ & {\mathcal C}_4 \ar@{.}@/^4pc/[uullllll] \ar@{-}[uu] \ar@{-}[uur]& {\mathcal C}_s(3) \ar@{-}[uuulll] \ar@{-}[uul]\\
           \ &\ & \bullet {\mathcal C}_3\ar@{.}[uu] & \ & \ & \ & \ & \ \\
           \ &\ & \ & \ & \bullet H_{1, 1}\ar@{-}[uuul]\ar@{-}[uuur] \ar@{-}[uurrr]& \ & \langle -\mathrm{id}\rangle \ar@{.}[uuullll] \ar@{.}[uu] \ar@{-}[uur]& \ \\
           \ &\ & \ & \ & \bullet \{\mathrm{id}\} \ar@{.}[uull] \ar@{-}[u] \ar@{.}[urr]& \ & \ & \ \\
}
\end{equation*}

In this diagram, the dotted inclusions correspond to those contained in $\mathrm{SL}_2(\mathbb{F}_3)$. This information will be relevant, as $\overline{\rho}_{E, 3}(G_K)\subset \SL_2(\mathbb{F}_3)$ if and only if $K\supset \mathbb{Q}(\sqrt{-3})$, the cyclotomic extension of cubic roots of unity (see e.g.\cite[Section 1.3]{Yvon}). Additionally, we have marked with the symbol $(\bullet)$ the subgroups that fix a nontrivial element (as justified below).

Let us review the groups in the diagram:

\normalsize
\begin{enumerate}
    \item $\mathrm{GL}_2(\mathbb{F}_3)$. It has cardinality $48$. It has no nontrivial fixed points.
    
    \item $\mathrm{SL}_2(\mathbb{F}_3)$. It has cardinality $24$. It has no nontrivial fixed points.
    
    \item $B(3)$, Borel subgroup. Isomorphic to $\mathcal{D}_{12}$, the dihedral group of $12$ elements, and has the form
    $$\begin{pmatrix} * & *\\ 0 & *\end{pmatrix}.$$ It has no nontrivial fixed points, since if $(x,y)^t$ were a fixed point, for every invertible matrix $\begin{pmatrix} a & b\\ 0 & d\\\end{pmatrix}$, it would have to satisfy
    \begin{equation*}
        \begin{pmatrix} a & b\\ 0 & d
        \end{pmatrix}\begin{pmatrix} x \\ y \end{pmatrix}=\begin{pmatrix} x \\ y \end{pmatrix},
    \end{equation*} implying $ax + by=x$, $dy=y$. Taking $d=2$, we obtain $y=0$. Therefore $ax=x$; taking $a=2$, we would have $x=0$.
    
    \item $N_{ns}(3)$, the non-split Cartan normalizer. This group is isomorphic to $\widetilde{ \mathcal{D}}_{16}$, a quasi-dihedral group of order 16. It can be written as
    $$
    \left\langle \begin{pmatrix} 1 & -1\\ 1 & 1\end{pmatrix}, \begin{pmatrix} 1 & 0\\ 0 & -1\end{pmatrix} \right\rangle.
    $$     
    It has no nontrivial fixed points, since if $(x,y)^t$ were a fixed point, in particular
    \begin{equation*}
        \begin{pmatrix} 1 & -1\\ 1 & 1\end{pmatrix}\begin{pmatrix}x \\y\end{pmatrix}=\begin{pmatrix}x \\y\end{pmatrix},
    \end{equation*} hence $x-y=x$, $x+y=y$, implying $x=y=0$.
    
    \item $Q_8$. It is the subgroup 
    $$
    \left\langle \begin{pmatrix} 0 & 1\\ -1 & 0\end{pmatrix},\begin{pmatrix} 1 & 1\\ 1 & -1\end{pmatrix} \right\rangle.
    $$ 
    This group has no nontrivial fixed points, since if $(x,y)^t$ were a fixed point, in particular
    \begin{equation*}
        \begin{pmatrix} 1 & 1\\ 1 & -1\end{pmatrix}\begin{pmatrix}x \\y\end{pmatrix}=\begin{pmatrix}x \\y\end{pmatrix},
    \end{equation*} hence $x+y=x$, $x-y=y$, implying $x=y=0$.

     \item $N_s(3)$. It is isomorphic to $\mathcal{D}_8$, with the form
    $$
    \begin{pmatrix} * & 0\\ 0 & *\end{pmatrix}\cup \begin{pmatrix} 0 & *\\ * & 0\end{pmatrix}.
    $$ 
    This group has no fixed points, since for every $a, b\neq 0$, it should satisfy that 
    \begin{equation*}
        \begin{pmatrix} a & 0\\ 0 & b\end{pmatrix}\begin{pmatrix}x \\y\end{pmatrix}=\begin{pmatrix}x \\y\end{pmatrix},
    \end{equation*} meaning $ax=x$, $by=y$, which is only possible if $x=y=0$.
    
    \item ${\mathcal C}_{ns}(3)$. It is a cyclic group isomorphic to $\mathcal{C}_8$, generated by 
    $$ 
    \begin{pmatrix} 1 & -1\\ 1 & 1\end{pmatrix}.
    $$ 
    Again, if $(x,y)^t$ was a fixed point, we would have
    \begin{equation*}
        \begin{pmatrix} 1 & -1\\ 1 & 1\end{pmatrix}\begin{pmatrix}x \\y\end{pmatrix}=\begin{pmatrix}x \\y\end{pmatrix},
    \end{equation*} implying $x-y=x$, $x + y=y$, hence $x=y=0$.
    
    \item $\mathcal{C}_6$. It is the subgroup can be written as 
    $$
    \left\langle \begin{pmatrix} -1 & -1\\ 0 & -1\end{pmatrix} \right\rangle.
    $$ 
    This group has no nontrivial fixed points, since if there was such point $(x,y)^t$, 
    \begin{equation*}
        \begin{pmatrix} -1 & -1\\ 0 & -1\end{pmatrix}\begin{pmatrix}x \\y\end{pmatrix}=\begin{pmatrix}x \\y\end{pmatrix},
    \end{equation*} hence $-x-y=x$, $-y=y$, implying $x=y=0$.
    \item $H_{3, 1}$. This case is isomorphic to $S_3$, of the form
    $$
    \begin{pmatrix} 1 & *\\ 0 & *\end{pmatrix}.
    $$ 
    In this case $(x,y)^t$ is a fixed point if and only if for all $a, b$ with $b\neq 0$, we have
    \begin{equation*}
        \begin{pmatrix} 1 & a\\ 0 & b\end{pmatrix}\begin{pmatrix}x \\y\end{pmatrix}=\begin{pmatrix}x \\y\end{pmatrix},
    \end{equation*} that is, $x + ay=x$, $by=y$. Therefore, the nontrivial fixed points are  $(1,0)^t$ and $(2,0)^t$.
    \item $H_{3, 2}$. Again this case is isomorphic to $S_3$ with
    $$
    \begin{pmatrix} * & *\\ 0 & 1\end{pmatrix}.
    $$  
    This group has no nontrivial fixed points, since if $(x,y)^t$ were a fixed point, in particular
    \begin{equation*}
        \begin{pmatrix} a & b\\ 0 & 1\end{pmatrix}\begin{pmatrix}x \\y\end{pmatrix}=\begin{pmatrix}x \\y\end{pmatrix},
    \end{equation*} for every $a\neq 0$, hence $ax + by=x$ for all $a=1, 2$, $b=0, 1, 2$. Taking $a=2, b=0$, we obtain $x=0$, and taking any $a$, $b=1$, we obtain $y=0$.
   
    \item $\mathcal{C}_4$. It is generated by the matrix 
    $$
    \begin{pmatrix} 0 & 1\\ -1 & 0\end{pmatrix}.
    $$ 
    It has no nontrivial fixed points, since 
    \begin{equation*}
        \begin{pmatrix} 0 & 1\\ -1 & 0\end{pmatrix}\begin{pmatrix}x \\y\end{pmatrix}=\begin{pmatrix}x \\y\end{pmatrix},
    \end{equation*} 
    implies $y=x$, $-x=y$, hence $x=y=0$.
    \item ${\mathcal C}_s(3)$. It is isomorphic to $\mathcal{C}_2 \times \mathcal{C}_2$, written as
    $$
    \begin{pmatrix} * & 0\\ 0 & *\end{pmatrix}.
    $$  
    It has no fixed points, since if  $(x,y)^t$ were a fixed point, we would have
    \begin{equation*}
        \begin{pmatrix} a & 0\\ 0 & b\end{pmatrix}\begin{pmatrix}x \\y\end{pmatrix}=\begin{pmatrix}x \\y\end{pmatrix},
    \end{equation*} for every $a, b$ such that $ab\neq 0$. In particular,   $ax=x$, $by=y$, for $a=b=-1$, hence $x=y=0$.
    \item $\mathcal{C}_3$. It is the subgroup generated by the transvection 
    $$
    \begin{pmatrix} 1 & 1\\ 0 &1\end{pmatrix},$$
    and $(x,y)^t$ is a fixed point if and only if 
    \begin{equation*}
        \begin{pmatrix} 1 & 1\\ 0 & 1\end{pmatrix}\begin{pmatrix}x \\y\end{pmatrix}=\begin{pmatrix}x \\y\end{pmatrix},
    \end{equation*} 
    implying $x+y=x$, $y=y$. Therefore, this group has nontrivial fixed points, precisely $(1,0)^t$ and $(2,0)^t$.
    \item $H_{1, 1}$. It is a subgroup isomorphic to $\mathcal{C}_2$,  
    $$
    \begin{pmatrix} 1 & 0\\ 0 &*\end{pmatrix}.
    $$  
    Now, $(x,y)^t$ is a fixed point if and only if 
    \begin{equation*}
        \begin{pmatrix} 1 & 0\\ 0 & a\end{pmatrix}\begin{pmatrix}x \\y\end{pmatrix}=\begin{pmatrix}x \\y\end{pmatrix},
    \end{equation*} for $a=1, 2$, meaning $x=x$, $ay=y$. Therefore, this group has nontrivial fixed points: $(1,0)^t$ and $(2,0)^t$.
    \item $\langle -\mathrm{id}\rangle$. This group has no fixed points.
    \item $\{\mathrm{id}\}$. All points are fixed points for this group.
\end{enumerate}

\

Now, let us examine the possibilities that lead to the growth of torsion. Note that, after choosing a suitable basis of $E[\ell]$, we can assume that $\overline{\rho}_{E, 3}(G_\mathbb{Q})$ coincides with a subgroup in the list above. If we have a subgroup $H\subset \overline{\rho}_{E, 3}(G_\mathbb{Q})$ without fixed points, then any subgroup of  $\overline{\rho}_{E, 3}(G_\mathbb{Q})$ that is conjugate (inside $\mathrm{GL}_2(\mathbb{F}_3)$) to it will also satisfy that it does not have any fixed points. Therefore $\overline{\rho}_{E, 3}(G_K)$ cannot be conjugate (inside $\mathrm{GL}_2(\mathbb{F}_3)$) to such an $H$. We now proceed to analyse each case.

\

We have two different scenarios: the first one is when $E_\tors(\mathbb{Q})[3]$ is trivial, and the other is when $E_\tors(\mathbb{Q})[3]\simeq \mathcal{C}_3$.

\ 

If $E_\tors(\mathbb{Q})$ is trivial, we need a group without fixed points, such that it has a subgroup of index $2$ with fixed points. However, the only groups with fixed points are $H_{3, 1}$, ${\mathcal C}_3$, $H_{1, 1}$, and $\{\mathrm{id}\}$. And:
\begin{itemize}
    \item $H_{3, 1}$ is contained in $B(3)$: It can happen that $\overline{\rho}_{E, 3}(G_{\mathbb{Q}})$ is conjugate to $B(3)$ and $\overline{\rho}_{E, 3}(G_K)$ is conjugate (inside $\mathrm{GL}_2(\mathbb{F}_3)$) to $H_{3, 1}$.
    \item $H_{1, 1}$ is contained in ${\mathcal C}_s(3)$: It can happen that $\overline{\rho}_{E, 3}(G_{\mathbb{Q}})$ is conjugate to ${\mathcal C}_s(3)$ and $\overline{\rho}_{E, 3}(G_K)$ is conjugate (inside $\mathrm{GL}_2(\mathbb{F}_3)$) to $H_{1, 1}$.
    \item ${\mathcal C}_3\subset {\mathcal C}_6$, but as ${\mathcal C}_6\subseteq \SL_2(\mathbb{F}_3)$, it does not occur as  an image of the Galois representation.
    \item The same goes for $\langle -\mathrm{id}\rangle$, which, although contains $\{ \mathrm{id}\}$ as a subgroup of index $2$, is contained in $\SL_2(\mathbb{F}_3)$.
\end{itemize}   

If $E_\tors(\mathbb{Q})[3] \simeq \mathcal{C}_3$, necessarily $\Gal(\mathbb{Q}(E[3])/\mathbb{Q})\simeq H_{3, 1}, H_{1, 1}, \mathcal{C}_3$. Of these, the only subgroup that has $\{\mathrm{id}\}$ as a subgroup of index $2$ is $H_{1, 1}$. Thus, we have the possibility $\Gal(\mathbb{Q}(E[3])/\mathbb{Q})$ conjugate to $ H_{1, 1}$ and  $\Gal(\mathbb{Q}(E[3])/K)= \{\mathrm{id}\}$. As $\mathrm{id}\in \SL_2(\mathbb{F}_3)$, it must happen that $K=\mathbb{Q}(\sqrt{-3})$.

\begin{example}
Now, let us see examples of curves with good reduction at $3$ where a new point of $3$-torsion appears over a quadratic field ramified at $3$:

\begin{enumerate}
\item Curve $19.a2$. It has Galois group $\Gal(\mathbb{Q}(E[3])/\mathbb{Q})\simeq H_{1, 1}$, which has a fixed point. That is, the torsion over $\mathbb{Q}$ is $\mathcal{C}_3$. Over the field $\mathbb{Q}(\sqrt{-3})$, it has a trivial Galois group, i.e., the torsion is $\mathcal{C}_3 \times \mathcal{C}_3$.

\item Curve $80.b1$. It has the Galois group $\Gal(\mathbb{Q}(E[3])/\mathbb{Q})\simeq B(3)$, which has no fixed points. The torsion over $\mathbb{Q}$ is $\mathcal{C}_2$, but over the field $\mathbb{Q}(\sqrt{3})$, it has Galois group isomorphic to $H_{3, 1}$; the torsion is $\mathcal{C}_6$.

Note that this example is a case where there is extra ramification at $3$, but the growth occurs over an extension that is not the cyclotomic one. We also present an example where the torsion of the original curve is trivial.

\item Curve $50.b1$. It has Galois group $\Gal(\mathbb{Q}(E[3])/\mathbb{Q})\simeq B(3)$, which has no fixed points. The torsion over $\mathbb{Q}$ is trivial. Over the field $\mathbb{Q}(\sqrt{-15})$, it has Galois group isomorphic to $H_{3, 1}$; the torsion is $\mathcal{C}_3$. The curve $176.a1$ is another example of this situation.

\item Curve $175.b3$. This curve has Galois group $\Gal(\mathbb{Q}(E[3])/\mathbb{Q})\simeq {\mathcal C}_s(3)$, which has no fixed points. The torsion over $\mathbb{Q}$ is trivial. Over the field $\mathbb{Q}(\sqrt{-15})$ it has Galois group isomorphic to $H_{1, 1}$.
\end{enumerate}
\end{example}

In addition, we can characterize the growth of the torsion when $E$ has good reduction at $3$ and the discriminant of $K$ is a multiple of $3$. First, we prove there are no curves whose torsion grows from $\mathcal{C}_1$ to $\mathcal{C}_9$ and which have good reduction at $3$.

\begin{proposition}\label{prop:C1_to_C9}
    Let $E/\mathbb{Q}$ be an elliptic curve such that $E_\tors(\mathbb{Q})$ is trivial and $E_\tors(K) =  \mathcal{C}_9$, with $K$ a quadratic field. Then $E$ has bad reduction at 3.
\end{proposition}
\begin{proof}
    Let us call $K = \mathbb{Q}(\sqrt{d})$ with $d$ squarefree. First we consider a minimal model 
    $$
    E_1: y^2 + a_1xy + a_3y = x^3 + a_2x^2 + a_4x + a_6.
    $$

    Applying the change of variables
    \begin{equation*}\begin{cases}
        x = \displaystyle \frac{x'}{4}, \\ \\
        y = \displaystyle \frac{y'}{8}-\frac{1}{2}(a_1\frac{x}{4} + a_3),
        \end{cases}
    \end{equation*}
    we obtain a $\mathbb{Q}$-isomorphic curve
    $$
    E_2:y^2 = x^3 + b_2x^2 + 8b_4x + 16b_6,
    $$
    where the relation between discriminants is $2^{12}\Delta_1 = \Delta_2$. Now, we consider the twisted curve
    $$
    E_d: dy^2= x^3+ b_2x^2 + 8b_4x + 16b_6
    $$
    and we can use the fact (\cite[Corollary 4]{GonTorI}) that if $n>1$ is an odd integer, then
    $$
    E(K)[n]\cong  E(\mathbb{Q})[n]\times E_d(\mathbb{Q})[n].
    $$
    Setting $n=9$, we obtain that $E_d$ has a rational point of order $9$ over $\mathbb{Q}$. The change of variables $\{y'=y/d^2, x'=x/d\}$ transforms $E_d$ into the curve 
    $$
    E_3 : y^2=x^3 +db_2x^2 + 8d^2b_4x + 16d^3b_6.
    $$
   Note that $E_3$ also has a rational point of order $9$ that we will call $P = (x_1,y_1)$.
    The twist from $E_2$ to $E_3$ is the change of variables
    $$
    x' = dx,\quad y'= d\sqrt{d} y,
    $$
    which gives us the relation $\Delta_1 2^{12} d^6 = \Delta_3 $.  Now we transform $E_3$ into its Tate normal form based on $P$ as we did in Subsection \ref{subsec:N=8y16}. In order to do so, we pass through the equation
    $$
    E_4: y^2 + \overline{a}_1xy + \overline{a}_3y = x^3 + \overline{a}_2x^2,
    $$
    with the relations
    $$
\def\arraystretch{1.5}
\left\{ \begin{array}{lcl} 
s & = &\displaystyle \frac{8d^2b_4 + 2db_2x_1 + 3x_1^2}{2y_1}\\

\overline{a}_1 & = & 2s,\\ 
\overline{a}_2 & = & -s^2+ 3x_1 +db_2, \\ 
\overline{a}_3 & = & 2y_1, \\ 
\Delta_4 & = & 2^{12}d^6 \Delta_1.
\end{array} \right.
$$

We have that $\overline{a}_2\not=0$ because $P$ does not have order $3$ (see \cite[V.5]{Knapp}). So we can make a change of variables to get the Tate normal form 
$$
    \mathcal{T}_{b,c}:  y^2+ (1-c)xy - by = x^3 - bx^2,
$$
with $b = - \overline{a}_2^3/\overline{a}_3^2, c= 1 - (\overline{a}_1\overline{a}_2)/\overline{a}_3$ and the relation between discriminants
$$
\Delta_{b,c} = 2^{12}\left(\frac{\overline{a}_2}{\overline{a}_3}\right)^{12} d^6\Delta_1.
$$
Because of $\overline{a}_3^2b = -\overline{a}_2^3$, we obtain the equation
\begin{equation}\label{eq:Discriminant_Tate_3}
\Delta_{b,c} \overline{a}_3^4 = 2^{12}b^4d^6\Delta.
\end{equation}
Using \cite[Example 4.6]{Husemoller}, there exists $t\in\mathbb{Q}$ such that 
    $$
\def\arraystretch{1.5}
\left\{ \begin{array}{lcl} 
b & = &\displaystyle (t-1)t^2(t^2-t+1),\\
c & = & (t-1)t^2,\\ 
\Delta_{b,c} & = & (t-1)^9t^9(t^2-t+1)^3(t^3-6t^2+3t+1).
\end{array} \right.
$$

Now, we write $t = p/q$ with $\gcd(p,q) = 1$. Therefore, Equation \eqref{eq:Discriminant_Tate_3} yields 
\begin{equation}\label{eq:Discriminant_3_pq}
2^{12}\Delta_1 d^6q^7 (p^2-pq+q^2) = \overline{a}_3^4(p-q)^5p(p^3-6qp^2+3q^2p +q^3).
\end{equation}

Now we divide the proof in the following three cases:

\ 

\noindent \fbox{\textbf{Case I: } $v_3(p) = 0, v_3(q) \neq 0$.} Taking the $3$-adic valuation in Equation \eqref{eq:Discriminant_3_pq} we get
$$
6v_3(d) + 7v_3(q) = 4v_3(y_1).
$$
Because of the Nagell-Lutz theorem (as explained in the subsection for the cases $N=8,16$), $y_1^2|d^62^{12}\Delta_1$, which implies that $v_3(y_1)\in\{0,1,2,3\}$ because $d$ is squarefree. For each of these values, the last equation gives us a contradiction, as the left side is strictly greater than the right side.

\

\noindent \fbox{\textbf{Case II: } $v_3(q) = 0, v_3(p)\neq 0$.} Taking the $3$-adic valuation in Equation \eqref{eq:Discriminant_3_pq} we get
\begin{equation*}
 6v_3(d) = 4v_3(y_1) + v_3(p).
\end{equation*}
As $d$ is squarefree, $v_3(d)\in\{0,1\}$. If $v_3(d) = 0$, the equation gives us a contradiction. So $v_3(d) = 1$ and previous equation becomes
\begin{equation}\label{eq:Discriminants_vpnotzero}
6 = 4v_3(y_1) + v_3(p).
\end{equation}
Therefore, $v_3(y_1)\in\{0,1\}$. On the other hand, we have the relations $-b\overline{a}_3^2=\overline{a}_2^3$ and $c= 1 - (\overline{a}_1\overline{a}_2)/\overline{a}_3$, with 
$$
b = \frac{(p-q)p^2(p^2-pq+q^2)}{q^5}, \qquad 
c = \frac{(p-q)p^2}{q^3}.
$$

Taking the $3$-adic valuation, we get 
$$
\def\arraystretch{1.5}
\begin{array}{lcl} 
3v_3(\overline{a}_2) & = & v_3(b) + 2v_3(y_1) \text{ and }  v_3(b)=2v_3(p), \\
v_3(c) & = & \displaystyle v_3\left(1-\frac{\overline{a}_1\overline{a}_2}{\overline{a}_3}\right) \text{ and }  v_3(c)=2v_3(p)>0.
\end{array} 
$$
From the two equations for $v_3(c)$ we obtain that $v_3(\overline{a}_1)  +v_3(\overline{a}_2) -v_3(y_1) = 0$.

Now we assume that $v_3(y_1) = 0$. Equation \eqref{eq:Discriminants_vpnotzero} yields $v_3(p)=6$, so from the two equations for $v_3(b)$ we obtain that $v_3(b)= 12$ and $v_3(\overline{a}_2) = 4$. The relation for $c$ is 
$$
0 = v_3(\overline{a}_1)  +v_3(\overline{a}_2) -v_3(y_1) = v_3(8d^2b_4 + 2db_2x_1 + 3x_1^2) + 4, 
$$
which is a contradiction. Therefore, $v_3(y_1)=1$. Now we follow the same reasoning. Equation \eqref{eq:Discriminants_vpnotzero} yields $v_3(p) = 2$ and the equations for $b$ imply that $v_3(b)=4$ and $v_3(\overline{a}_2) = 2$. Finally, the relation for $c$ gives us
$$
0 = v_3(\overline{a}_1) + v_3(\overline{a}_2) - 1 = v_3(8d^2b_4 + 2db_2x_1+3x_1^2) >0,
$$
a contradiction.

\ 

\noindent \fbox{\textbf{Case III: } $v_3(p) = v_3(q) = 0$.} Taking the $3$-adic valuation in Equation \eqref{eq:Discriminant_3_pq}, we obtain
\begin{equation}\label{eq:valuation_Discriminant_3_pq}
6v_3(d) + v_3(p^2-pq + q^2) = 4v_3(y_1) + 5v_3(p-q) + v_3(p^3+ q^3 - 6p^2q + 3pq^2).
\end{equation}
 We can rewrite \eqref{eq:valuation_Discriminant_3_pq} in terms of $t = p/q$ and we get
\begin{equation}\label{eq:valuation_Discriminant_3_t}
6v_3(d) + v_3(t^2-t + 1) = 4v_3(y_1) + 5v_3(t-1) + v_3(t^3  - 6t^2 + 3t+1).
\end{equation}
As $v_3(t) = 0$, $t\equiv \pm 1\mod 3$. Let us assume first that $t = 1 + 3k$ for some $k\in \mathbb{Q}$ with $v_3(k)\geq 0$. Substituting $t$, we get
\begin{eqnarray*}
    t^2-t+1&=& 9k^2 + 3k + 1,\\
    t^3-6t^2+3t+1&=&343k^3 + 147k^2 - 42k - 17.
\end{eqnarray*}
which implies that $v_3(t^2-t+1) = v_3(t^3  - 6t^2 + 3t+1) = 0$. Therefore, \eqref{eq:valuation_Discriminant_3_t} becomes
$$
6v_3(d) = 4v_3(y_1) + 5v_3(t-1).
$$
Using that $v_3(t-1)>0$ and $v_3(d)\in\{0,1\}$ is easy to reach a contradiction. Then $t = -1 + 3k$ substituting again
\begin{eqnarray*}
    t^2-t+1&=& 9k^2 - 9k + 3,\\
    t^3-6t^2+3t+1&=&27k^3 - 81k^2 + 54k - 9.
\end{eqnarray*}
Therefore, $v_3(t^3  - 6t^2 + 3t+1) = 2$ and $v_3(t^2-t+1) = 1$. So \eqref{eq:valuation_Discriminant_3_t} becomes 
$$
6v_3(d) = 4v_3(y_1)  + 1,
$$
which implies that $1$ is even, a contradiction.
\end{proof}

\begin{proposition}\label{caso N=3}
    Let $E/\mathbb{Q}$ be an elliptic curve and $K=\mathbb{Q}(\sqrt{d})$ be a quadratic extension such that there exists a point $P\in E[3]$ such that $P\in E(K)\setminus E(\mathbb{Q})$. Let us assume $3|d$ and $E$ has good reduction at 3, then
    $$
    E_\tors(K) =   E_\tors(\mathbb{Q})\times \mathcal{C}_3.
    $$
\end{proposition}

\begin{proof}
If $E(\mathbb{Q})[3] =  \mathcal{C}_3$, the hypothesis implies that $E_\tors(K)=  E_\tors(\mathbb{Q})\times\mathcal{C}_3$, because of Theorem \ref{GT}. Therefore we can assume that $E(\mathbb{Q})[3]$ is trivial.
    
    \

Let us consider a minimal Weierstrass equation for $3$ with integer coefficients. 
$$
E: y^2 + a_1xy + a_3y=x^3 + a_2x^2 + a_4x + a_6.
$$
Now we go by contradiction. Using Theorem \ref{GT}, we see that the possible cases that can arise are: 
\begin{itemize}
    \item $E_\tors(\mathbb{Q})=\mathcal{C}_1$ and $E_\tors(K)=\mathcal{C}_9$,
    \item $E_\tors(\mathbb{Q})=\mathcal{C}_2$ and $E_\tors(K)=\mathcal{C}_{12}$ or $\mathcal{C}_2\times \mathcal{C}_6$, 
    \item $E_\tors(\mathbb{Q})=\mathcal{C}_4$ and $E_\tors(K)= \mathcal{C}_2\times \mathcal{C}_{12}$,
    \item $E_\tors(\mathbb{Q})=\mathcal{C}_2\times \mathcal{C}_2$ and $E_\tors(K)=\mathcal{C}_2\times \mathcal{C}_{12}$.
\end{itemize}

The first case is ruled out by Proposition \ref{prop:C1_to_C9}. We divide the proof into three cases:

\

\noindent \fbox{\textbf{Case I:} $E_\tors(\mathbb{Q})[2] = \mathcal{C}_2$ and $\mathcal{C}_4\leq E_\tors(K)$.} As in subsection \ref{caso N=4} we can do a change of variables and we get a Weierstrass equation:
$$
y^2 = x^3 + Ax^2 + Bx
$$
with $A,B\in \mathbb{Z}$ and the following equation between discriminants
\begin{equation*}
2^{8}\Delta = B^2(A^2-4B).
\end{equation*}
Now Lemma \ref{lema C2} tell us that $K = \mathbb{Q}(\sqrt{A\pm 2s})$ with $s\in\mathbb{Q}$ such that  $s^2 = B$, so $s\in\mathbb{Z}$. As $3|d$, $A\pm 2s$ must be a multiple of $3$. Then, from the equation relating the discriminants, we obtain
$$
0 = 2v_3(B) + v_3(A+2s) + v_3(A-2s) \geq 1,
$$
a contradiction.

\ 

\noindent \fbox{\textbf{Case II:} $E_\tors(\mathbb{Q})[2] = \mathcal{C}_2$ and $\mathcal{C}_2\times\mathcal{C}_2\leq E_\tors(K)$.} We can compute, as in Section 2, the $2$-torsion points with the $2$-division polynomial $\psi_2 = 2y + a_1x + a_3$; substituting 
$$
y \longmapsto \frac{1}{2}(-a_1x-a_3),
$$ 
we obtain the expression (see, again \cite[Ch. III]{Silverman})
$$
0=4x^3 + b_2x^2 + 2b_4x + b_6,
$$
where the discriminant of the polynomial 
$$
4^4(\alpha- \beta)^2(\alpha - c)^2(\beta- c)^2=16 \Delta.
$$ 

The $x$-coordinates of the nontrivial $2$-torsion points are the three roots of this polynomial. As there is a single rational $2$-torsion point, the above polynomial factors as
$$
4(x-\alpha)(x-\beta)(x-c),
$$
where $c\in \mathbb{Q}$ and $\alpha, \beta$ are conjugate elements in $\mathbb{Q}(\sqrt{d})$ by hypothesis. Let us say
$$
\alpha = a + b\sqrt{d}, \qquad \beta = a-b\sqrt{d}.
$$
    
As $3|d$, $\mathbb{Q}(\sqrt{d})$ ramifies at $3$, there is only one valuation of $K$ over $v_3$ such that $v_3(3) = 1$, so we will call it $v_3$ too. Therefore,
$$
\Delta=4^2(\alpha- \beta)^2(\alpha - c)^2(\beta- c)^2 \; \Longrightarrow \; 
0 = 2v_3(\alpha - \beta) + 2v_3(\alpha-c) + 2v_3(\beta-c). 
$$
Since $\alpha,\beta$ are roots of a rational polynomial whose leading coefficient is not divisible by $3$, every term in the right-hand side is nonnegative and $v_3(\alpha - \beta) = v_3(2b\sqrt{d})>0$, so we have a contradiction.
    
\
    
\noindent \fbox{\textbf{Case III:} $E_\tors(\mathbb{Q})[2] = \mathcal{C}_2\times \mathcal{C}_2$ and $\mathcal{C}_4\times\mathcal{C}_2\leq E_\tors(K)$.} Again, as in Section \ref{caso N=4}, we can do a change of variables and we get a Weierstrass equation:
$$
y^2 = x^3 + Ax^2 + Bx
$$
with $A,B\in \mathbb{Z}$ and the following equation between discriminants:
\begin{equation*}
2^{8}\Delta = B^2(A^2-4B).
\end{equation*}
In this situation, like in Section \ref{caso N=4}, we can assume that there is a point $Q\in E(K)[4]$ such that $2Q=(0, 0)$. Applying Lemma \ref{lema del Knapp} as in Subsection \ref{caso N=4}, we can assume that the polynomial  $x^3 + Ax^2 + Bx$ factors as $x(x-\alpha)(x-\beta)$, where $\alpha, \beta\in \mathbb{Q}$ satisfy that $-\alpha$ and  $-\beta$ are squares in $K^*$, but they are not both squares in $\mathbb{Q}^*$.  Thus either $-\alpha=a^2$ and $\-\beta=db^2$, or $-\alpha=da^2$ and $\-\beta=b^2$, or else $-\alpha=da^2$ and $\-\beta=db^2$ for some rational integers $a$ and $b$. Replacing $B=\alpha\beta$ and $A=-\alpha-\beta$ in the equation for $\Delta$ above, we obtain in all cases that there is an integer $V\in\mathbb{Z}$ such that
$$
2^{8}\Delta = dV.
$$
As $3\nmid \Delta$, we get a contradiction.
\end{proof}

\section{The primes $\ell=5$ and $\ell=7$}\label{sec:l57}

Throught this section we can already assume that the curve $E$ does not have complex multiplication, as the CM case has already been solved by E. González--Jiménez \cite{GJ}. Our goal here is to prove Theorem \ref{l57}, as stated in the introduction.

\

\noindent \textbf{Theorem \ref{l57}.} \textit{Let $E/\mathbb{Q}$ be an elliptic curve, $K$ a quadratic number field such that $E_\tors(\mathbb{Q}) \neq E_\tors(K)$. If $p=5,7$ ramifies in $K$, then $p \vert N_E$.}

\ 

The idea to exclude ramification at the primes $\ell=5$ and $\ell=7$ will be that, except in the case of good ordinary reduction with the action of the wild inertia being trivial, there cannot be $\ell$-torsion points defined over a quadratic extension of $\mathbb{Q}$ (nor over $\mathbb{Q}$). This is because the inertia group at $\ell$ is already too large to have fixed points. We will formalize this shortly.

Next, we will have to deal with the case where there is good ordinary reduction and the action of the wild inertia is trivial. In this case, it may happen that there is an $\ell$-torsion point defined over $\mathbb{Q}$ (and also over a quadratic field), and a detailed analysis of the possible images of $\overline{\rho}_{E, \ell}(G_{\mathbb{Q}})$ will be necessary.

\ 

Let $E/\mathbb{Q}$ be an elliptic curve with good reduction at the prime $\ell$, for $\ell = 5$ or $\ell = 7$. Let $K/\mathbb{Q}$ be a quadratic extension, $P\in E[\ell]$ a point such that $P\in E(K)\setminus E(\mathbb{Q})$. As mentioned in the first section, this implies $K\subset \mathbb{Q}(E[\ell])$. 

Suppose $\ell$ ramifies in $K/\mathbb{Q}$. Take a prime $\Lambda\vert \ell$ of $\mathbb{Q}(E[\ell])$, and consider the inertia group $I(\Lambda/\ell)\subset \Gal(\mathbb{Q}(E[\ell])/\mathbb{Q})$. We denote $\lambda=K\cap \Lambda$.

\ 

Let $L=\mathbb{Q}(E[\ell])^{I(\Lambda/\ell)}$ be the fixed field of $\mathbb{Q}(E[\ell])$ by the action of $I(\Lambda/\ell)$. The extension $L/\mathbb{Q}$ is, by definition, not ramified in $\Lambda$, while $K/\mathbb{Q}$ is totally ramified in $\ell$. Therefore they are linearly disjoint over $\mathbb{Q}$. We have the following diagram:

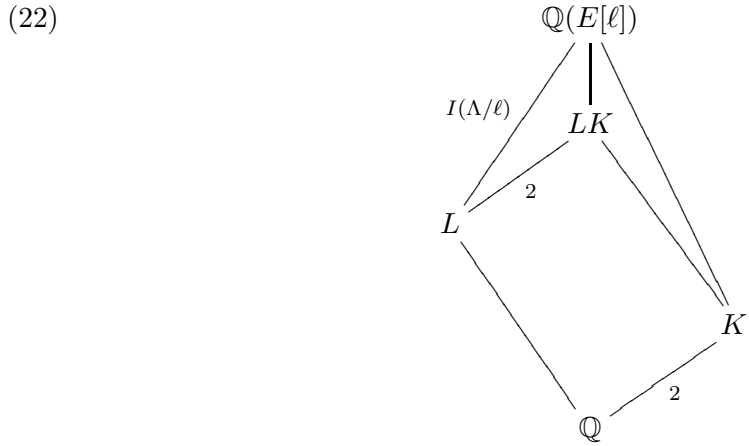
\begin{equation}\label{eq:diagrama}
 \xymatrix{\ & \mathbb{Q}(E[\ell]) \ar@{-}[d] \ar@{-}[dddr] \ar@{-}[ddl]_{I(\Lambda/\ell)}& \ \\
           \ & LK\ar@{-}[dl]^2 \ar@{-}[ddr]& \ \\
           L \ar@{-}[ddr] & \ & \  \ \\
           \  & \ & K \ar@{-}[dl]^2  \ \\
           \ & \mathbb{Q} & \ \\}
\end{equation}

Let us observe that the group $I(\Lambda/\ell)$ coincides with $\overline{\rho}_{E, \ell}(I_{\ell})$, where $I_{\ell}\subset G_{\mathbb{Q}}$ is the inertia group at $\ell$, after fixing a decomposition group at $\ell$ compatible with the prime $\Lambda$ of $\mathbb{Q}(E[\ell])$. Let us also note that $\Gal(\mathbb{Q}(E[\ell])/LK)\subset \Gal(\mathbb{Q}(E[\ell])/K)$, which fixes a point of $\ell$-torsion.

Therefore, we have:

\begin{itemize}
    \item The image $\overline{\rho}_{E, \ell}(G_{\mathbb{Q}})$ contains the subgroup $I(\Lambda/\ell)\simeq \overline{\rho}_{E, \ell}(I_\ell)$.
    
    \item $I(\Lambda/\ell)$ contains a subgroup of index 2, $\Gal(\mathbb{Q}(E[\ell])/LK)$, which fixes at least one point of $\ell$-torsion.
\end{itemize}

Next, we will see the form of $I(\Lambda/\ell)$ according to the type of reduction of $E$. 
The following propositions, which are proved in \cite[Section 1.11, Section 1.12]{Serre}, exactly determine the image of $I_{\ell}$ under $\overline{\rho}_{E, \ell}$:

\begin{proposition}\label{prop:action_inertia} 
Let $E/\mathbb{Q}$ be an elliptic curve with good reduction of height 1 or multiplicative reduction in a prime $\ell$. Then one and only one of the following possibilities holds:
\begin{enumerate}
    \item The wild inertia group $I^{\wild}_\ell$ acts trivially on $E[\ell]$. Then the image of $I_{\ell}$ is a cyclic group of order $\ell - 1$. In a suitable basis, it coincides with the subgroup
    \begin{equation*}
        H_1=\left\{\begin{pmatrix} a & 0\\ 0 & 1\end{pmatrix}: a\in \mathbb{F}_{\ell}^{\times}\right\}.
    \end{equation*}
    
    \item The wild inertia group $I^\wild_\ell$ does not act trivially on $E[\ell]$. Then the image of $I^\wild_\ell$ under $\overline{\rho}_{E, \ell}$ is a cyclic group of order $\ell$, and in a suitable basis, it can be represented as
    \begin{equation*}
        H_2 = \left\{\begin{pmatrix} 1 & b\\ 0 & 1\end{pmatrix}: b\in \mathbb{F}_{\ell}\right\}.
    \end{equation*}
    The image of $I_\ell$ has order $\ell(\ell - 1)$ and can be represented as
    \begin{equation*}
        H_3 = \left\{\begin{pmatrix} a & b\\ 0 & 1\end{pmatrix}: a\in \mathbb{F}_{\ell}^{\times}, b\in \mathbb{F}_{\ell}\right\}.
    \end{equation*}
\end{enumerate}
\end{proposition}

\begin{proposition}\label{prop:12}
Let $p,\ell$ be two different primes. Let $E/\mathbb{Q}$ be an elliptic curve with  multiplicative reduction in a prime $p$. Then the image of $I_{\ell}$ is trivial or a cyclic group of order $\ell$.
\end{proposition}

 We will now analyze each of the possible types of reduction of $E$ at $\ell$, proving that, for each case, the existence of a point $P\in E(K)[\ell]\setminus E(\mathbb{Q})$ yields a contradiction. Taking Remark \ref{rem:1} into account, this reasoning will prove Theorem \ref{l57}.

\subsection{Ordinary  good reduction  or multiplicative reduction, with nontrivial wild inertia action} 

 We will assume in this case that $E$ has good reduction at $\ell$ and  $I^\wild_\ell$ does not act trivially. Note that $H_3$ does not leave any element other than $(0, 0)^t$ invariant; in particular, there cannot be any rational point of $\ell$-torsion.

Furthermore, no subgroup of index $2$ in $H_3$ leaves any element other than $(0, 0)^t$ invariant.  Indeed, since the cardinality of $H_3$ is $\ell(\ell-1)$ and $\ell>2$, any subgroup $H$ of $H_3$ of index $2$ must contain an element of order $\ell$. Such elements are precisely the non-trivial elements of $H_2$. Therefore, $H$ contains $H_2$. On the other hand, we have
\begin{equation*}
    \begin{pmatrix} 1 & b\\ 0 & 1\end{pmatrix}\begin{pmatrix} x \\ y\end{pmatrix} = \begin{pmatrix} x \\ y\end{pmatrix}
\end{equation*}
if and only if $by=0$, so the only nontrivial fixed points by the entire subgroup are of the form $(x, 0)^t$, where $x\neq 0$.  However, any subgroup of index $2$ must contain at least one matrix of the form 
$$
\begin{pmatrix} a & b\\ 0 & 1\end{pmatrix}, \mbox{ with } a\neq 1,
$$ (otherwise the subgroup would be contained in $H_2$, and $H_2$ is already too small to have index two in $H_3$, provided $\ell>3$). 
Now, this matrix does not fix the point $(x, 0)^t$. This shows that for $\ell=5, 7$, it is not possible to add an  $\ell$-torsion point over a quadratic extension that ramifies at $\ell$.

\

\subsection{Ordinary good reduction  or multiplicative reduction, with trivial wild inertia action}\label{subsec:wild_trivial_inertia}

Now suppose that $E/\mathbb{Q}$ is a curve with good ordinary reduction at $\ell$, such that the wild inertia group acts trivially. By Proposition \ref{prop:action_inertia}, we have that $\overline{\rho}_{E, \ell}(G_{\mathbb{Q}})$ contains, in a suitable basis, a subgroup of the form 
\begin{equation*}
    H_1= \left\{\begin{pmatrix} a & 0\\ 0 & 1\end{pmatrix}: a\in \mathbb{F}_{\ell}^{\times}\right\}.
\end{equation*}

Sutherland \cite{Sutherland} and Zywina \cite{Zywina} have characterized all possible subgroups that can appear as $\overline{\rho}_{E, \ell}(G_{\mathbb{Q}})$ in the cases $\ell=5$ and $\ell=7$. We will go through each of the cases, ruling out in each one the possibility of having a point of $\ell$-torsion over a quadratic extension $K/\mathbb{Q}$.

\

\noindent \fbox{\textbf{Case 5.2.I:} $\ell=5$.} In this case, \cite{Zywina} shows that $\overline{\rho}_{E, 5}(G_{\mathbb{Q}})$ is conjugate in $\GL_2(\mathbb{F}_5)$ to a group from a list of 15 possible groups (see \cite[Theorem 1.4]{Zywina}). Furthermore, Sutherland studies these groups and determines in each case the index of the largest subgroup that fixes a nonzero vector in $\mathbb{F}_{5}^2$; this quantity coincides with the degree of the minimal extension $K/\mathbb{Q}$ such that $E$ has a rational point of $5$-torsion. 

In \cite[Table 3, p. 64]{Sutherland}, we can find the list of these 15 groups along with their indices. There are only four of them where this degree is $2$, specifically those labeled as 5Cs.1.3, 5Cs.4.1, 5B.1.4, and 5B.4.1. Let us examine each of these cases:

\begin{itemize}
    \item 5Cs.1.3. It is a cyclic subgroup of order $4$, generated by the matrix 
    $$
    \begin{pmatrix} 3 & 0\\ 0 & 4\end{pmatrix}.
    $$
    Since we are assuming that $E$ has good ordinary reduction at $5$, $\overline{\rho}_{E, \ell}(G_{\mathbb{Q}})$ must contain a subgroup conjugate to $H_1$, which is also cyclic of order $4$. However, the subgroup 5Cs.1.3 is not conjugate to $H_1$, because it does not fix any nonzero element of $\mathbb{F}_5^2$. The conclusion is that this image cannot occur if $E$ has good ordinary reduction at $5$ (it also cannot have supersingular reduction, by a similar reasoning). 
    
    \item 5Cs.4.1. It is a group of order $8$, generated by the matrices
    $$
    \left\{ 
    \begin{pmatrix} 4 & 0\\ 0 & 4\end{pmatrix}, \begin{pmatrix} 1 & 0\\ 0 & 2\end{pmatrix} \right\}.
    $$
    The subgroup generated by the matrix 
    $$
    \begin{pmatrix} 1 & 0\\ 0 & 2\end{pmatrix}
    $$ 
    is conjugate to the subgroup $H_1$, and it is the only cyclic subgroup of order $4$ in 5Cs.4.1 that fixes a nonzero element of $\mathbb{F}_5^2$. Therefore,
    $$
    \Gal(\mathbb{Q}(E[5])/L)=\overline{\rho}_{E, 5}(I_{5})= \left\langle \begin{pmatrix} 1 & 0\\ 0 & 2\end{pmatrix} \right\rangle.
    $$
    Moreover, $\Gal(\mathbb{Q}(E[5])/K)$ is a subgroup of $\Gal(\mathbb{Q}(E[5])/\mathbb{Q})$ of order $4$. Since $L\neq K$, it must  be a subgroup of $\Gal(\mathbb{Q}(E[5])/\mathbb{Q})$ different from $\Gal(\mathbb{Q}(E[5])/L)$, and thus it cannot fix a nonzero element of $\mathbb{F}_5^2$. 
    
    Therefore, $\Gal(\mathbb{Q}(E[5])/K)$ does not have nontrivial fixed points, which contradicts the fact that $E$ has a nontrivial $5$-torsion point over $K$.
    
    \item 5B.1.4. It is a subgroup of order $20$, generated by the matrices
    \begin{equation*}
        \left\{ \begin{pmatrix} 4 & 0\\ 0 & 3\end{pmatrix}, \begin{pmatrix} 1 & 1\\ 0 & 1\end{pmatrix} \right\}.
    \end{equation*}
    It is easy to verify that no cyclic subgroup of order $4$ has a nontrivial fixed point. Therefore, there is no subgroup conjugate to $H_1$. This implies that the curve $E/\mathbb{Q}$ cannot have good reduction at $5$. 

\item 5B.4.1. is a subgroup of order $40$, generated by the matrices
 \begin{equation*}
\left\{ \begin{pmatrix} 4 & 0\\ 0 & 4\end{pmatrix}, \begin{pmatrix} 1 & 0\\ 0 & 2\end{pmatrix}, \begin{pmatrix} 1 & 1\\ 0 & 1\end{pmatrix} \right\}.
\end{equation*}

Again, there is a subgroup which is conjugate to $H_1$, the subgroup 
\begin{equation*} 
H=
\left\{
\begin{pmatrix} 1 & 0\\ 0 & 2\end{pmatrix}, \begin{pmatrix} 1 & 0\\ 0 & 4\end{pmatrix}, \begin{pmatrix} 1 & 0\\ 0 & 3\end{pmatrix}, \begin{pmatrix} 1 & 0\\ 0 & 1\end{pmatrix}
\right\}.
\end{equation*} 

In this case there are more cyclic subgroups of order $4$ that have a fixed point; specifically:
\begin{equation*} 
\left\{
\begin{pmatrix} 1 & 1\\ 0 &  2\end{pmatrix},    \begin{pmatrix} 1  & 3\\  0&  4\end{pmatrix}, \begin{pmatrix} 1 &  2\\  0 & 3\end{pmatrix}, \begin{pmatrix} 1 & 0\\ 0 & 1\end{pmatrix}
\right\}, \quad
 \left\{
\begin{pmatrix} 1 & 1\\ 0 & 3\end{pmatrix}, \begin{pmatrix} 1 & 4\\  0 & 4\end{pmatrix}, \begin{pmatrix} 1 &  3 \\ 0 & 2\end{pmatrix},  \begin{pmatrix} 1 & 0\\ 0 & 1\end{pmatrix}
\right\},
\end{equation*}

\begin{equation*}
 \left\{
\begin{pmatrix} 1 & 2\\ 0 &  2\end{pmatrix}, \begin{pmatrix} 1 &  1\\  0 &  4\end{pmatrix},  \begin{pmatrix} 1 & 4  \\  0 & 3\end{pmatrix} \begin{pmatrix} 1 & 0\\ 0 & 1\end{pmatrix}
\right\}, \quad
 \left\{
\begin{pmatrix} 1 & 3\\ 0 & 3\end{pmatrix}, \begin{pmatrix} 1 &  2\\  0 & 4\end{pmatrix}, \begin{pmatrix} 1 &  4\\  0 & 2\end{pmatrix},
\begin{pmatrix} 1 & 0\\ 0 & 1\end{pmatrix}
\right\}.
\end{equation*}

In fact, all of these groups are conjugate to $H$  inside the group $5B.4.1$:

\begin{equation*}
 \begin{pmatrix} 1 & 1 \\ 0 & 1\end{pmatrix} \begin{pmatrix} 1 & 0\\ 0 & 2\end{pmatrix} \begin{pmatrix}1 & 1\\ 0 & 1\end{pmatrix}^{-1}=\begin{pmatrix} 1 & 1\\ 0 & 2\end{pmatrix}, \quad
 \begin{pmatrix} 1 & 1 \\ 0 & 2\end{pmatrix} \begin{pmatrix} 1 & 0\\ 0 & 3\end{pmatrix} \begin{pmatrix}1 & 1\\ 0 & 2\end{pmatrix}^{-1}=\begin{pmatrix} 1 & 1\\ 0 & 3\end{pmatrix},
\end{equation*}

\begin{equation*}
 \begin{pmatrix} 1 & 2 \\ 0 & 1\end{pmatrix} \begin{pmatrix} 1 & 0\\ 0 & 2\end{pmatrix} \begin{pmatrix}1 & 2\\ 0 & 1\end{pmatrix}^{-1}=\begin{pmatrix} 1 & 2\\ 0 & 2\end{pmatrix}, \quad
 \begin{pmatrix} 1 & 3 \\ 0 & 2\end{pmatrix} \begin{pmatrix} 1 & 0\\ 0 & 3\end{pmatrix} \begin{pmatrix}1 & 3\\ 0 & 2\end{pmatrix}^{-1}=\begin{pmatrix} 1 & 3\\ 0 & 3\end{pmatrix}.
\end{equation*}

By choosing a basis, we can assume that $\Gal(\mathbb{Q}(E[5])/L)$ is one of them.
If there exists a point of $5$-torsion over $K$, the only possibility is that $\Gal(\mathbb{Q}(E[5])/K)$, which is a group of order $20$, is the union of all elements whose upper left entry is $1$. This can be seen by looking at the list of elements of the group 5B.4.1; the only 20 elements that fix the same element of $\mathbb{F}_5^2$ are these.

But in that case, $\Gal(\mathbb{Q}(E[5])/L)\subset \Gal(\mathbb{Q}(E[5])/K)$, and therefore $K\subset L$. Since we have seen that $K$ and $L$ are linearly disjoint over $\mathbb{Q}$, this is a contradiction.
\end{itemize}

\ 

\noindent \fbox{\textbf{Case 5.2.II:} $\ell=7$.} Again, \cite{Zywina} studies this case and shows that $\overline{\rho}_{E, 7}(G_{\mathbb{Q}})$ is conjugate in $\GL_2(\mathbb{F}_7)$ to a group from a list of 16 possible groups (see \cite[Theorem 1.5]{Zywina}). Sutherland studies these groups and determines in each case the index of the largest subgroup that fixes a nonzero vector in $\mathbb{F}_{7}^2$; this quantity coincides with the degree of the minimal extension $K/\mathbb{Q}$ such that $E$ has a rational point of $7$-torsion. 

In \cite[Table 3, p.65]{Sutherland} we can find the list of these 16 groups along with these indices. There are only two of them where this degree is $2$, specifically those labeled as 7B.1.6, 7B.6.1. Let us examine each of these cases:

\begin{itemize}
    \item 7B.1.6. It is a group of order 42, generated by the matrices
    \begin{equation*}
        \left\{ \begin{pmatrix} 6 & 0\\ 0 & 4\end{pmatrix}, \begin{pmatrix} 1 & 1\\ 0 & 1\end{pmatrix} \right\}.
    \end{equation*}
    It can be verified that no cyclic subgroup of order 6 fixes a nonzero element of $\mathbb{F}_7^2$. Therefore, $E$ cannot have good reduction at $\ell=7$. 
    
    \item 7B.6.1. It is a group of order 84, generated by 
    \begin{equation*}
        \left\{ \begin{pmatrix} 6 & 0\\ 0 & 6\end{pmatrix},  \begin{pmatrix} 1 & 0\\ 0 & 3\end{pmatrix},\begin{pmatrix} 1 & 1\\ 0 & 1\end{pmatrix} \right\}.
    \end{equation*}
    Again, we can calculate which cyclic subgroups of order 6 fix a point. We obtain the following:

 \begin{equation*}
\left\{
\begin{pmatrix} 1 &  0\\ 0 & 3\end{pmatrix}, \begin{pmatrix} 1 &  0\\ 0 & 2\end{pmatrix},  \begin{pmatrix} 1 &  0 \\  0 & 6\end{pmatrix}, \begin{pmatrix} 1 &  0\\ 0 & 4\end{pmatrix}, \begin{pmatrix}  1 & 0\\ 0 & 5\end{pmatrix}, \begin{pmatrix} 1 & 0\\ 0 & 1\end{pmatrix} \right\},
 \end{equation*}

 \begin{equation*}
\left\{
\begin{pmatrix} 1 &  1\\ 0 & 3\end{pmatrix}, \begin{pmatrix} 1 &  4\\ 0 & 2\end{pmatrix},  \begin{pmatrix} 1 &  6 \\  0 & 6\end{pmatrix}, \begin{pmatrix} 1 &  5\\ 0 & 4\end{pmatrix}, \begin{pmatrix}  1 & 2\\ 0 & 5\end{pmatrix}, \begin{pmatrix} 1 & 0\\ 0 & 1\end{pmatrix} \right\},
 \end{equation*} 
 
 \begin{equation*}
\left\{
\begin{pmatrix} 1 &  3\\ 0 & 3\end{pmatrix}, \begin{pmatrix} 1 &  5\\ 0 & 2\end{pmatrix},  \begin{pmatrix} 1 &  4 \\  0 & 6\end{pmatrix}, \begin{pmatrix} 1 &  1\\ 0 & 4\end{pmatrix}, \begin{pmatrix}  1 & 6\\ 0 & 5\end{pmatrix}, \begin{pmatrix} 1 & 0\\ 0 & 1\end{pmatrix} \right\},
 \end{equation*} 

 \begin{equation*}
\left\{
\begin{pmatrix} 1 &  2\\ 0 & 3\end{pmatrix}, \begin{pmatrix} 1 &  1\\ 0 & 2\end{pmatrix},  \begin{pmatrix} 1 &  5 \\  0 & 6\end{pmatrix}, \begin{pmatrix} 1 &  3\\ 0 & 4\end{pmatrix}, \begin{pmatrix}  1 & 4\\ 0 & 5\end{pmatrix}, \begin{pmatrix} 1 & 0\\ 0 & 1\end{pmatrix} \right\},
 \end{equation*}

 \begin{equation*}
\left\{
\begin{pmatrix} 1 &  4\\ 0 & 3\end{pmatrix}, \begin{pmatrix} 1 &  2\\ 0 & 2\end{pmatrix},  \begin{pmatrix} 1 &  3 \\  0 & 6\end{pmatrix}, \begin{pmatrix} 1 &  6\\ 0 & 4\end{pmatrix}, \begin{pmatrix}  1 & 1\\ 0 & 5\end{pmatrix}, \begin{pmatrix} 1 & 0\\ 0 & 1\end{pmatrix} \right\},
 \end{equation*}  
 
 \begin{equation*}
\left\{
\begin{pmatrix} 1 &  5\\ 0 & 3\end{pmatrix}, \begin{pmatrix} 1 &  6\\ 0 & 2\end{pmatrix},  \begin{pmatrix} 1 &  2 \\  0 & 6\end{pmatrix}, \begin{pmatrix} 1 &  4\\ 0 & 4\end{pmatrix}, \begin{pmatrix}  1 & 3\\ 0 & 5\end{pmatrix}, \begin{pmatrix} 1 & 0\\ 0 & 1\end{pmatrix} \right\},
 \end{equation*} 

 \begin{equation*}
\left\{
\begin{pmatrix} 1 &  6\\ 0 & 3\end{pmatrix}, \begin{pmatrix} 1 &  3\\ 0 & 2\end{pmatrix},  \begin{pmatrix} 1 &  1 \\  0 & 6\end{pmatrix}, \begin{pmatrix} 1 &  2\\ 0 & 4\end{pmatrix}, \begin{pmatrix}  1 & 5\\ 0 & 5\end{pmatrix}, \begin{pmatrix} 1 & 0\\ 0 & 1\end{pmatrix} \right\}.
 \end{equation*} 

It is easy to check that they are all conjugate, and in any case $\Gal(\mathbb{Q}(E[7])/L)$ must be one of them (once we have chosen a basis).

Looking at the list of elements of the group 7B.6.1, we again come to the conclusion that if $\Gal(\mathbb{Q}(E[7])/K)$ is a subgroup of order $42$ that fixes an element, it must necessarily consist of all the elements that have a $1$ in the upper left entry. But then $\Gal(\mathbb{Q}(E[7])/L)\subset \Gal(\mathbb{Q}(E[7])/K)$ and therefore $K\subset L$, which is not possible because they are linearly disjoint over $\mathbb{Q}$.
\end{itemize}

\subsection{Supersingular reduction} In this case, the image by $\overline{\rho}_{E, \ell}$ of the inertia group $I_{\ell}$ is an non-split Cartan subgroup (cf.~\cite[Section 1.9]{Serre}). Consider a matrix 
$$
\begin{pmatrix} a & b\varepsilon\\ b & a\end{pmatrix}
$$ 
in this group, where $\varepsilon$ is a non-quadratic residue modulo $\ell$. Then, if $(x, y)^t$ is a fixed point of this matrix, different from $(0, 0)^t$, the following system of equations holds:
\begin{equation*}
    (a-1) x + b\varepsilon y = 0, \quad bx + (a-1) y = 0.
\end{equation*}
If $a=1$ and $b\neq 0$, then $b\varepsilon y=0$ and $bx=0$, which implies $x=y=0$, contradicting the assumption that $(x, y)^t\not=(0, 0)^t$. If $a\neq 1$, then $y=-bx/(a-1)$, and substituting into the other equation, we have
$$ 
x \Big((a-1)^2 -b^2\varepsilon \Big)=0.
$$ 
Since $\varepsilon$ is not a quadratic residue modulo $\ell$, $(a-1)^2 - b^2\varepsilon\neq 0$,  
thus $x=0$, and hence $y=0$. 

In other words, these matrices (with the exception of the identity matrix) do not have fixed points different from $(0, 0)^t$. Since there is at least one such matrix  in any subgroup of index $2$, we conclude that $E(K)$ does not contain points from $E[\ell]$. Note that there cannot be rational points of $\ell$-torsion either.

\subsection{Main result for $\ell=5,7$.} This last argument finishes the proof of Theorem \ref{l57}. However, our techniques allow us to prove a stronger version of this result, as stated in the introduction (and including some additional information on $\ell=3$).

\
    
\noindent \textbf{Theorem \ref{addred}.} 
\textit{Let $E/\mathbb{Q}$ be an elliptic curve. Assume $K$ is a quadratic number field such that there exists $P\in E(K)[\ell]\setminus E(\mathbb{Q})$, with $\ell\geq 3$ prime.  
    \begin{enumerate}
        \item For every prime $p\neq \ell$ that ramifies in $K$, $E$ has additive reduction at $p$, i.e. $p^2|N_E$.
        \item If $p=\ell>3$  ramifies in $K$, $E$ has additive reduction at $p$, i.e. $p^2|N_E$.
    \end{enumerate}}

\begin{proof}
    By hypothesis, 
$$
\mathbb{Q} \subsetneq K \cap \mathbb{Q}(E[\ell]) \subset K \quad \text{and} \quad [K:\mathbb{Q}] = 2.
$$ 
The existence of $P$ implies that $K\cap \mathbb{Q}(E[\ell])\not=\mathbb{Q}$, thus $K \subset \mathbb{Q}(E[\ell])$, which implies that $p$ ramifies in $\mathbb{Q}(E[\ell])$. 

Let us prove the first part of the theorem. We have that, by hypothesis, $p\not=\ell$. From the Néron–Ogg–Shafarevich criterion (see \cite[VII, Theorem 7.1]{Silverman}), we can conclude that $E$ has bad reduction at $p$. By way of contradiction, suppose that $E$ has multiplicative reduction at $p$.  

\ 

Let $\mathfrak{P}|p$ be a prime above $p$ in $\mathbb{Q}(E[\ell])$. Let $I(\mathfrak{P}|p)$ denote the inertia group associated with this prime, and let 
$$
L = \mathbb{Q}(E[\ell])^{I(\mathfrak{P}|p)}
$$ 
be the maximal unramified subextension at $p$ of $\mathbb{Q}(E[l])|\mathbb{Q}$. Since $p$ ramifies in $K$, we have the following diagram:
\[
\begin{tikzcd}
                                                                      & {\mathbb{Q}(E[\ell])} \arrow[rddd, no head]                  &                        \\
                                                                      & LK \arrow[u, no head]                                     &                        \\
L \arrow[ru, "2"', no head] \arrow[ruu, "I(\mathfrak{P}|p)", no head] &                                                           &                        \\
                                                                      &                                                           & K \arrow[luu, no head] \\
                                                                      & \mathbb{Q} \arrow[luu, no head] \arrow[ru, "2"', no head] &                       
\end{tikzcd}
\]

By Proposition \ref{prop:12}, $I(\mathfrak{P}|p)$ has order $\ell\geq 3$. Since $\ell > 2$, the diagram yields a contradiction. 

The second part of the theorem was already proven in this section.
\end{proof}

\section{Conclusion and final remarks}

Merging together Proposition \ref{prop:2}, Theorem \ref{l2} (proven in Sections \ref{sec:2_strict} and \ref{sec:2_mixed}), and Theorem \ref{l57} (proven in Section \ref{sec:l57}), we can state the following result:

\begin{theorem}
    Let $E/\mathbb{Q}$ be an elliptic curve with conductor $N_E$ and $K = \mathbb{Q} (\sqrt{d})$ a quadratic number field with $E_\tors(\mathbb{Q}) \neq E_\tors(K)$. Then if $p \in \mathbb{Z}$ is a prime such that $p|d$, then either $p|N_E$ or $p=3$. 
\end{theorem}

For $p>3$ (that is, Theorem \ref{l57}), this theorem has already been proved with different techniques and in a more general context by Mentzelos Melistas (see \cite[Theorem 1.5]{Melistas}). 

Also, when $p>2$, the condition that $E$ has a $p$-torsion point in $E(K)\setminus E(\mathbb{Q})$, where $K=\mathbb{Q}(\sqrt{d})$, is equivalent to the condition that the quadratic twist $E_d$ has a rational $p$-torsion point. As one of the referees pointed out, one could try to prove Theorem 1 for $p>3$ relying on this fact. And, indeed, by makes use of some results of Olson \cite{Olson}, one can find an alternative proof in this case.

\

However, for $p>3$, the techniques used in the alternative proof given in Section 5 allowed us to prove Theorem \ref{addred}, which is a stronger result.

\ 

While exploring the distinctive case $\ell = 3$ we have also been able to offer some interesting remarks which give us a fuller picture of the phenomenon. Mainly:

\

\noindent \textbf{Proposition \ref{caso N=3}.}
\textit{Let $E/\mathbb{Q}$ be an elliptic curve and $K=\mathbb{Q}(\sqrt{d})$ be a quadratic extension such that there exists a point $P\in E[3]$ satisfying that $P\in E(K)\setminus E(\mathbb{Q})$. Let us assume $3|d$ and $E$ has good reduction at 3, then}
$$
E_\tors(K) =   E_\tors(\mathbb{Q})\times \mathcal{C}_3.
$$

\

As a final note, we must underscore the fact that the ultimate problem of shortlisting the quadratic fields where the torsion grows in terms of invariants (of the curve and the quadratic field alike) still should admit many improvements. 

In this sense, our main result is just a step in the direction of sieving the set of suitable quadratic extensions and future work by the authors is already in progress concerning these matters.

\section{Appendix: Matrix groups appearing in Section \ref{sec:l57}}

\begin{itemize}

\item Group 5B.1.4:

$$
\left\{ 
\begin{pmatrix} 1 & 0\\ 0 & 1\end{pmatrix}, \;
\begin{pmatrix} 1 & 0\\ 0 & 4\end{pmatrix}, \;
\begin{pmatrix} 1 & 1\\ 0 & 1\end{pmatrix}, \;
\begin{pmatrix} 1 & 1\\ 0 & 4\end{pmatrix}, \;
\begin{pmatrix} 1 & 2\\ 0 & 1\end{pmatrix}, \;
\begin{pmatrix} 1 & 2\\ 0 & 4\end{pmatrix}, \;
\begin{pmatrix} 1 & 3\\ 0 & 1\end{pmatrix}, \;
\begin{pmatrix} 1 & 3\\ 0 & 4\end{pmatrix}, \right.
$$
$$
\begin{pmatrix} 1 & 4\\ 0 & 1\end{pmatrix}, \;
\begin{pmatrix} 1 & 4\\ 0 & 4\end{pmatrix}, \;
\begin{pmatrix} 4 & 0\\ 0 & 2\end{pmatrix}, \; 
\begin{pmatrix} 4 & 0\\ 0 & 3\end{pmatrix}, \; 
\begin{pmatrix} 4 & 1\\ 0 & 2\end{pmatrix}, \; 
\begin{pmatrix} 4 & 1\\ 0 & 3\end{pmatrix}, \; 
\begin{pmatrix} 4 & 2\\ 0 & 2\end{pmatrix}, \; 
\begin{pmatrix} 4 & 2\\ 0 & 3\end{pmatrix},
$$
$$
\left. \begin{pmatrix} 4 & 3\\ 0 & 2\end{pmatrix}, \
\begin{pmatrix} 4 & 3\\ 0 & 3\end{pmatrix}, \;
\begin{pmatrix} 4 & 4\\ 0 & 2\end{pmatrix}, \;
\begin{pmatrix} 4 & 4\\ 0 & 3\end{pmatrix} \right\}.
$$

\ 

\item  Group 5B.4.1:

$$
\left\{
\begin{pmatrix} 1 & 0\\ 0 & 1\end{pmatrix}, \; 
\begin{pmatrix} 1 & 0\\ 0 & 2\end{pmatrix}, \;
\begin{pmatrix} 1 & 0\\ 0 & 3\end{pmatrix}, \;
\begin{pmatrix} 1 & 0\\ 0 & 4\end{pmatrix}, \;
\begin{pmatrix} 1 & 1\\ 0 & 1\end{pmatrix}, \;
\begin{pmatrix} 1 & 1\\ 0 & 2\end{pmatrix}, \;
\begin{pmatrix} 1 & 1\\ 0 & 3\end{pmatrix}, \;
\begin{pmatrix} 1 & 1\\ 0 & 4\end{pmatrix}, \right. 
$$
$$
\begin{pmatrix} 1 & 2\\ 0 & 1\end{pmatrix}, \;
\begin{pmatrix} 1 & 2\\ 0 & 2\end{pmatrix}, \;
\begin{pmatrix} 1 & 2\\ 0 & 3\end{pmatrix}, \;
\begin{pmatrix} 1 & 2\\ 0 & 4\end{pmatrix}, \;
\begin{pmatrix} 1 & 3\\ 0 & 1\end{pmatrix}, \;
\begin{pmatrix} 1 & 3\\ 0 & 2\end{pmatrix}, \;
\begin{pmatrix} 1 & 3\\ 0 & 3\end{pmatrix}, \;
\begin{pmatrix} 1 & 3\\ 0 & 4\end{pmatrix},
$$
$$
\begin{pmatrix} 1 & 4\\ 0 & 1\end{pmatrix}, \;
\begin{pmatrix} 1 & 4\\ 0 & 2\end{pmatrix}, \;
\begin{pmatrix} 1 & 4\\ 0 & 3\end{pmatrix}, \; 
\begin{pmatrix} 1 & 4\\ 0 & 4\end{pmatrix}, \;
\begin{pmatrix} 4 & 0\\ 0 & 1\end{pmatrix}, \;
\begin{pmatrix} 4 & 0\\ 0 & 2\end{pmatrix}, \;
\begin{pmatrix} 4 & 0\\ 0 & 3\end{pmatrix}, \; 
\begin{pmatrix} 4 & 0\\ 0 & 4\end{pmatrix}, 
$$
$$
\begin{pmatrix} 4 & 1\\ 0 & 1\end{pmatrix}, \;
\begin{pmatrix} 4 & 1\\ 0 & 2\end{pmatrix}, \;
\begin{pmatrix} 4 & 1\\ 0 & 3\end{pmatrix}, \;
\begin{pmatrix} 4 & 1\\ 0 & 4\end{pmatrix}, \;
\begin{pmatrix} 4 & 2\\ 0 & 1\end{pmatrix}, \; 
\begin{pmatrix} 4 & 2\\ 0 & 2\end{pmatrix}, \;
\begin{pmatrix} 4 & 2\\ 0 & 3\end{pmatrix}, \;
\begin{pmatrix} 4 & 2\\ 0 & 4\end{pmatrix}, 
$$
$$
\left. \begin{pmatrix} 4 & 3\\ 0 & 1\end{pmatrix}, \; 
\begin{pmatrix} 4 & 3\\ 0 & 2\end{pmatrix}, \;
\begin{pmatrix} 4 & 3\\ 0 & 3\end{pmatrix}, \;
\begin{pmatrix} 4 & 3\\ 0 & 4\end{pmatrix}, \;
\begin{pmatrix} 4 & 4\\ 0 & 1\end{pmatrix}, \;
\begin{pmatrix} 4 & 4\\ 0 & 2\end{pmatrix}, \;
\begin{pmatrix} 4 & 4\\ 0 & 3\end{pmatrix}, \; 
\begin{pmatrix} 4 & 4\\ 0 & 4\end{pmatrix} 
\right\}.
$$

\ 

\item  Group 7B.1.6:

$$
\left\{
\begin{pmatrix} 1 & 0\\ 0 & 1\end{pmatrix}, \;
\begin{pmatrix} 1 & 0\\ 0 & 2\end{pmatrix}, \;
\begin{pmatrix} 1 & 0\\ 0 & 4\end{pmatrix}, \;
\begin{pmatrix} 1 & 1\\ 0 & 1\end{pmatrix}, \;
\begin{pmatrix} 1 & 1\\ 0 & 2\end{pmatrix}, \;
\begin{pmatrix} 1 & 1\\ 0 & 4\end{pmatrix}, \;
\begin{pmatrix} 1 & 2\\ 0 & 1\end{pmatrix}, \;
\begin{pmatrix} 1 & 2\\ 0 & 2\end{pmatrix}, \right.
$$
$$
\begin{pmatrix} 1 & 2\\ 0 & 4\end{pmatrix}, \;
\begin{pmatrix} 1 & 3\\ 0 & 1\end{pmatrix}, \;
\begin{pmatrix} 1 & 3\\ 0 & 2\end{pmatrix}, \;
\begin{pmatrix} 1 & 3\\ 0 & 4\end{pmatrix}, \;
\begin{pmatrix} 1 & 4\\ 0 & 1\end{pmatrix}, \;
\begin{pmatrix} 1 & 4\\ 0 & 2\end{pmatrix}, \; 
\begin{pmatrix} 1 & 4\\ 0 & 4\end{pmatrix}, \;
\begin{pmatrix} 1 & 5\\ 0 & 1\end{pmatrix},
$$
$$
\begin{pmatrix} 1 & 5\\ 0 & 2\end{pmatrix}, \;
\begin{pmatrix} 1 & 5\\ 0 & 4\end{pmatrix}, \;
\begin{pmatrix} 1 & 6\\ 0 & 1\end{pmatrix}, \;
\begin{pmatrix} 1 & 6\\ 0 & 2\end{pmatrix}, \;
\begin{pmatrix} 1 & 6\\ 0 & 4\end{pmatrix}, \;
\begin{pmatrix} 6 & 0\\ 0 & 1\end{pmatrix}, \;
\begin{pmatrix} 6 & 0\\ 0 & 2\end{pmatrix}, \;
\begin{pmatrix} 6 & 0\\ 0 & 4\end{pmatrix},
$$
$$
\begin{pmatrix} 6 & 1\\ 0 & 1\end{pmatrix}, \;
\begin{pmatrix} 6 & 1\\ 0 & 2\end{pmatrix}, \;
\begin{pmatrix} 6 & 1\\ 0 & 4\end{pmatrix}, \;
\begin{pmatrix} 6 & 2\\ 0 & 1\end{pmatrix}, \;
\begin{pmatrix} 6 & 2\\ 0 & 2\end{pmatrix}, \;
\begin{pmatrix} 6 & 2\\ 0 & 4\end{pmatrix}, \;
\begin{pmatrix} 6 & 3\\ 0 & 1\end{pmatrix}, \;
\begin{pmatrix} 6 & 3\\ 0 & 2\end{pmatrix},
$$
$$
\begin{pmatrix} 6 & 3\\ 0 & 4\end{pmatrix}, \;
\begin{pmatrix} 6 & 4\\ 0 & 1\end{pmatrix}, \;
\begin{pmatrix} 6 & 4\\ 0 & 2\end{pmatrix}, \;
\begin{pmatrix} 6 & 4\\ 0 & 4\end{pmatrix}, \;
\begin{pmatrix} 6 & 5\\ 0 & 1\end{pmatrix}, \;
\begin{pmatrix} 6 & 5\\ 0 & 2\end{pmatrix}, \;
\begin{pmatrix} 6 & 5\\ 0 & 4\end{pmatrix}, \;
\begin{pmatrix} 6 & 6\\ 0 & 1\end{pmatrix}, \;
$$
$$
\left. \begin{pmatrix} 6 & 6\\ 0 & 2\end{pmatrix}, \;
\begin{pmatrix} 6 & 6\\ 0 & 4\end{pmatrix} \right\}.
$$

\ 

\item Group 7B.6.1:

$$
\left\{ 
\begin{pmatrix} 1 & 0\\ 0 & 1\end{pmatrix}, \;
\begin{pmatrix} 1 & 0\\ 0 & 2\end{pmatrix}, \;
\begin{pmatrix} 1 & 0\\ 0 & 3\end{pmatrix}, \;
\begin{pmatrix} 1 & 0\\ 0 & 4\end{pmatrix}, \;
\begin{pmatrix} 1 & 0\\ 0 & 5\end{pmatrix}, \;
\begin{pmatrix} 1 & 0\\ 0 & 6\end{pmatrix}, \;
\begin{pmatrix} 1 & 1\\ 0 & 1\end{pmatrix}, \;
\begin{pmatrix} 1 & 1\\ 0 & 2\end{pmatrix}, \right.
$$
$$
\begin{pmatrix} 1 & 1\\ 0 & 3\end{pmatrix}, \;
\begin{pmatrix} 1 & 1\\ 0 & 4\end{pmatrix}, \;
\begin{pmatrix} 1 & 1\\ 0 & 5\end{pmatrix}, \;
\begin{pmatrix} 1 & 1\\ 0 & 6\end{pmatrix}, \;
\begin{pmatrix} 1 & 2\\ 0 & 1\end{pmatrix}, \;
\begin{pmatrix} 1 & 2\\ 0 & 2\end{pmatrix}, \;
\begin{pmatrix} 1 & 2\\ 0 & 3\end{pmatrix}, \;
\begin{pmatrix} 1 & 2\\ 0 & 4\end{pmatrix},
$$
$$
\begin{pmatrix} 1 & 2\\ 0 & 5\end{pmatrix}, \;
\begin{pmatrix} 1 & 2\\ 0 & 6\end{pmatrix}, \;
\begin{pmatrix} 1 & 3\\ 0 & 1\end{pmatrix}, \;
\begin{pmatrix} 1 & 3\\ 0 & 2\end{pmatrix}, \;
\begin{pmatrix} 1 & 3\\ 0 & 3\end{pmatrix}, \;
\begin{pmatrix} 1 & 3\\ 0 & 4\end{pmatrix}, \;
\begin{pmatrix} 1 & 3\\ 0 & 5\end{pmatrix}, \;
\begin{pmatrix} 1 & 3\\ 0 & 6\end{pmatrix},
$$
$$
\begin{pmatrix} 1 & 4\\ 0 & 1\end{pmatrix}, \;
\begin{pmatrix} 1 & 4\\ 0 & 2\end{pmatrix}, \;
\begin{pmatrix} 1 & 4\\ 0 & 3\end{pmatrix}, \;
\begin{pmatrix} 1 & 4\\ 0 & 4\end{pmatrix}, \;
\begin{pmatrix} 1 & 4\\ 0 & 5\end{pmatrix}, \;
\begin{pmatrix} 1 & 4\\ 0 & 6\end{pmatrix}, \;
\begin{pmatrix} 1 & 5\\ 0 & 1\end{pmatrix}, \;
\begin{pmatrix} 1 & 5\\ 0 & 2\end{pmatrix},
$$
$$
\begin{pmatrix} 1 & 5\\ 0 & 3\end{pmatrix}, \;
\begin{pmatrix} 1 & 5\\ 0 & 4\end{pmatrix}, \;
\begin{pmatrix} 1 & 5\\ 0 & 5\end{pmatrix}, \;
\begin{pmatrix} 1 & 5\\ 0 & 6\end{pmatrix}, \;
\begin{pmatrix} 1 & 6\\ 0 & 1\end{pmatrix}, \;
\begin{pmatrix} 1 & 6\\ 0 & 2\end{pmatrix}, \;
\begin{pmatrix} 1 & 6\\ 0 & 3\end{pmatrix}, \;
\begin{pmatrix} 1 & 6\\ 0 & 4\end{pmatrix},
$$
$$
\begin{pmatrix} 1 & 6\\ 0 & 5\end{pmatrix}, \;
\begin{pmatrix} 1 & 6\\ 0 & 6\end{pmatrix}, \;
\begin{pmatrix} 6 & 0\\ 0 & 1\end{pmatrix}, \;
\begin{pmatrix} 6 & 0\\ 0 & 2\end{pmatrix}, \;
\begin{pmatrix} 6 & 0\\ 0 & 3\end{pmatrix}, \;
\begin{pmatrix} 6 & 0\\ 0 & 4\end{pmatrix}, \;
\begin{pmatrix} 6 & 0\\ 0 & 5\end{pmatrix}, \;
\begin{pmatrix} 6 & 0\\ 0 & 6\end{pmatrix},
$$
$$
\begin{pmatrix} 6 & 1\\ 0 & 1\end{pmatrix}, \;
\begin{pmatrix} 6 & 1\\ 0 & 2\end{pmatrix}, \;
\begin{pmatrix} 6 & 1\\ 0 & 3\end{pmatrix}, \;
\begin{pmatrix} 6 & 1\\ 0 & 4\end{pmatrix}, \;
\begin{pmatrix} 6 & 1\\ 0 & 5\end{pmatrix}, \;
\begin{pmatrix} 6 & 1\\ 0 & 6\end{pmatrix}, \;
\begin{pmatrix} 6 & 2\\ 0 & 1\end{pmatrix}, \;
\begin{pmatrix} 6 & 2\\ 0 & 2\end{pmatrix},
$$
$$
\begin{pmatrix} 6 & 2\\ 0 & 3\end{pmatrix}, \;
\begin{pmatrix} 6 & 2\\ 0 & 4\end{pmatrix}, \;
\begin{pmatrix} 6 & 2\\ 0 & 5\end{pmatrix}, \;
\begin{pmatrix} 6 & 2\\ 0 & 6\end{pmatrix}, \;
\begin{pmatrix} 6 & 3\\ 0 & 1\end{pmatrix}, \;
\begin{pmatrix} 6 & 3\\ 0 & 2\end{pmatrix}, \; 
\begin{pmatrix} 6 & 3\\ 0 & 3\end{pmatrix}, \;
\begin{pmatrix} 6 & 3\\ 0 & 4\end{pmatrix},
$$
$$
\begin{pmatrix} 6 & 3\\ 0 & 5\end{pmatrix}, \;
\begin{pmatrix} 6 & 3\\ 0 & 6\end{pmatrix}, \;
\begin{pmatrix} 6 & 4\\ 0 & 1\end{pmatrix}, \;
\begin{pmatrix} 6 & 4\\ 0 & 2\end{pmatrix}, \;
\begin{pmatrix} 6 & 4\\ 0 & 3\end{pmatrix}, \;
\begin{pmatrix} 6 & 4\\ 0 & 4\end{pmatrix}, \;
\begin{pmatrix} 6 & 4\\ 0 & 5\end{pmatrix}, \;
\begin{pmatrix} 6 & 4\\ 0 & 6\end{pmatrix}, 
$$
$$
\begin{pmatrix} 6 & 5\\ 0 & 1\end{pmatrix}, \;
\begin{pmatrix} 6 & 5\\ 0 & 2\end{pmatrix}, \;
\begin{pmatrix} 6 & 5\\ 0 & 3\end{pmatrix}, \;
\begin{pmatrix} 6 & 5\\ 0 & 4\end{pmatrix}, \;
\begin{pmatrix} 6 & 5\\ 0 & 5\end{pmatrix}, \;
\begin{pmatrix} 6 & 5\\ 0 & 6\end{pmatrix}, \;
\begin{pmatrix} 6 & 6\\ 0 & 1\end{pmatrix}, \;
\begin{pmatrix} 6 & 6\\ 0 & 2\end{pmatrix},
$$
$$
\left. \begin{pmatrix} 6 & 6\\ 0 & 3\end{pmatrix}, \;
\begin{pmatrix} 6 & 6\\ 0 & 4\end{pmatrix}, \;
\begin{pmatrix} 6 & 6\\ 0 & 5\end{pmatrix}, \;
\begin{pmatrix} 6 & 6\\ 0 & 6\end{pmatrix} \right\}.
$$
\end{itemize}

\noindent \textbf{Conflict of Interest:} Not Applicable.

\bibliographystyle{amsplain}
\bibliography{QuadTor}

\end{document}